\documentclass[11pt,reqno]{amsart}
\usepackage{amssymb,amsfonts,amsthm,amscd,stmaryrd,dsfont,esint,upgreek,constants,todonotes}
\usepackage{pdfsync}
\usepackage{amsmath, constants}\multlinegap=0pt
\usepackage[latin1]{inputenc}
\usepackage{amsthm}
\usepackage{amssymb}
\usepackage[francais,english]{babel}

\newconstantfamily{C}{symbol=C}
\newconstantfamily{c}{symbol=c}

\usepackage{color}

\usepackage{amsfonts,amsmath,layout}

\newcommand{\N}{\mathbb N}

\newcommand{\Z}{\mathbb Z}

\newcommand{\R}{\mathbb R}

\newcommand{\T}{\mathbb{T}}

\def\R{\mathbb R}
\def\N{\mathbb N}
\def\Z{\mathbb Z}

\def\P{\mathbb P}

\def\ep{\epsilon}
\def\rg{\rangle} 
\def\lg{\langle}

\newcommand{\be}{\begin{equation}}
\newcommand{\ee}{\end{equation}}

\def\1{{\bf 1}}

\def\inte{\int_{Q_1}}
\def\dive{{\rm div}}

\def\ds{\displaystyle}

\newtheorem{Theorem}{Theorem}[section]
\newtheorem{Definition}[Theorem]{Definition}
\newtheorem{Proposition}[Theorem]{Proposition}
\newtheorem{Lemma}[Theorem]{Lemma}

\newtheorem{Remark}[Theorem]{Remark}
\newtheorem{Remarks}[Theorem]{Remarks}

\begin{document}

\title{Weak solutions for first order mean field games with local coupling}
\author{Pierre Cardaliaguet}
\address{Ceremade, Universit\'e Paris-Dauphine,
Place du Maréchal de Lattre de Tassigny, 75775 Paris cedex 16 - France}
\email{cardaliaguet@ceremade.dauphine.fr }

\dedicatory{Version: \today}

\maketitle

\begin{abstract} Existence and uniqueness of a weak solution for first order mean field game systems with local coupling are obtained by variational methods. This solution can be used to devise $\epsilon-$Nash equilibria for deterministic  differential games with a finite (but large) number of players. For smooth data, the first component of the weak solution of the MFG system is proved to satisfy (in a viscosity sense) a time-space degenerate elliptic differential equation. 
\end{abstract}

\section*{Introduction}

This paper is devoted to the analysis of  first order  mean field game systems with a local coupling. The general form of these systems is: 
\be\label{MFG}
\left\{\begin{array}{cl}
(i)&- \partial_t \phi +H(x,D\phi) =f(x,m(x,t))\\
(ii) & \partial_t m-{\rm div} (mD_pH(x,D\phi))=0\\
(iii)& m(0)=m_0, \; \phi(x,T)=\phi_T(x)
\end{array}\right.
\ee
where the Hamiltonian $H:\R^d\times \R^d\to\R$ is convex in the second variable, the coupling
$f:\R^d\times [0,+\infty)\to [0,+\infty)$ is increasing with respect to the second variable, $m_0$ is a probability density and $\phi_T:\R^d\to\R$ is a given function.  
 In \eqref{MFG}, the scalar unknowns  $(\phi,m)$ are defined on $[0,T]\times \R^d$ and $f$ is a coupling between the two equations. These systems are used to formalize deterministic differential games with an infinite number of players.  The function $\phi$ can be understood as the value function---for a typical and small player who controls his velocity---of  a finite horizon optimal control problem  in which the density $m$ of the other players enters as a datum through the coupling $f$. For this optimal control problem, the optimal feedback of this small player is then (formally) given by the vector field $-D_pH(x,D\phi(t,x))$. When all players play according to this rule, their distribution density $m=m(t,x)$ evolves in time by the continuity equation \eqref{MFG}-(ii). Note that the HJ equation is backward in time (with a terminal condition), while the continuity equation is forward in time (with an initial condition). 

Mean field game systems have been introduced simultaneously by Lasry and Lions  \cite{LL06cr1, LL06cr2, LL07mf} and by Huang, Caines and Malham\'e \cite{HCMieeeAC06}. For second order MFG systems (i.e., systems containing a nondegenerate diffusion) or for first order MFG systems in which the coupling $f$ is of nonlocal nature and regularizing, structure conditions ensuring existence and uniqueness of solutions are well-understood (see, in particular, the discussions in \cite{LL06cr2, LL07mf}). For first order systems in which the coupling is local---as above---the picture is not so clear. Under specific structure conditions, one can expect to have smooth solutions  \cite{LLperso}: the idea is to transform the system into a quasilinear elliptic equation in time space for $\phi$. A priori estimates are then obtained by Bernstein method. The approach however requires some smoothness on the coefficients $H$ and $f$ and some structure conditions  (typically $f(x,m)=\log(m)$) ensuring that the measure $m$ does not vanish. 

Here we work in a different setting: we require only mild regularity on the coefficients and  the measure $m$ can vanish. Under suitable coercivity conditions on $H$ and $f$, we show that system \eqref{MFG} has a unique weak solution (Theorem \ref{theo:mainex}), which is moreover stable with respect to the data (Proposition \ref{Prop:stabilo}). For simplicity we work with space periodic boundary conditions (i.e., in the torus $\T^d=\R^d/\Z^d$). Our result requires that $H=H(x,p)$ is strictly convex and has a super-linear growth with respect to $p$, while $f=f(x,m)$ is strictly increasing in $m$ with a growth depending on the growth of $H$.  Moreover we impose $f$ to be continuous at $m=0$, which prevents our result to apply to couplings of the form $f(x,m)=\ln(m)$. By a weak solution $(m,\phi)$ of \eqref{MFG}, we roughly mean that  $\phi$ is  continuous while $m$ is integrable, that  \eqref{MFG}-(i)  holds a.e. while \eqref{MFG}-(ii) is to be understood in the sense of distribution (see Definition \ref{def:weaksolMFG}). 

Our starting point is the idea---introduced by Lasry and Lions in \cite{LL07mf}---that the MFG system can be understood as an optimality condition for two problems in duality. The first one is an optimal control problem for a Hamilton-Jacobi  equation: one controls the state variable  $\phi$ by a distributed control $\alpha:(0,T)\times \T^d\to \R$ in order to minimize the criterium
$$
 \int_0^T\int_{\T^d} F^*\left(x,\alpha(t,x) \right)\ dxdt - \int_{\T^d} \phi(0,x)dm_0(x).
$$
The state $\phi$ is driven by the backward HJ equation
$$
\left\{ \begin{array}{l}
 -\partial_t\phi(t,x)+H(x,D\phi(t,x))= \alpha(t,x)\qquad {\rm in}\; (0,T)\times \T^d\\
 \phi(T,x)=\phi_T(x)\qquad {\rm in}\;  \T^d
\end{array}\right.
$$
In the above problems, $F^*$ denotes the Fenchel conjugate of the primitive $F$ of $f=f(x,m)$ with respect to the variable $m$. 
The second  control problem is an optimal control problem for a continuity equation: one now controls the state variable $m$ through a vector field $v:(0,T)\times \T^d\to \R^d$ in order to minimize the quantity 
$$
 \int_0^T\int_{\T^d} m(t,x) H^*\left(x, -v(t,x)\right)+ F(x,m(t,x)) \ dxdt + \int_{\T^d} \phi_T(x)m(T,x)dx,
$$
where $m$ solves the continuity equation
$$
\partial_t m+{\rm div} (mv)=0\; {\rm in}\; (0,T)\times \T^d, \qquad m(0)=m_0. 
 $$
In the above expression, $H^*$ is the Fenchel conjugate of the Hamiltonian $H=H(x,\xi)$ with respect to the second variable $\xi$,  $F$ still being the primitive of $f=f(x,m)$ with respect to the variable $m$.

Our first contributions consist in giving a meaning to the above optimal control problems, in showing that they have a solution (at least when properly relaxed) and in proving that, if $\alpha$ is optimal in the first problem, with associate state $\phi$, and if $v$ is optimal in the second problem, with associate state $m$, then the pair $(m,\phi)$ is the unique weak solution of the MFG system \eqref{MFG}. 

This approach  is reminiscent of several works on optimal transport problems, which also involve a coupling between a HJ equation and a continuity equation (and in particular the so-called Benamou-Brenier approach \cite{bb}). However, in most papers on the subject, the HJ equation does not depend on $m$. This is also the case of a model in geometric optics analyzed by Gosse and James \cite{JG} and by Ben Moussa and Kossioris \cite{BMK}, where, furthermore, the HJ equation is forward in time: the analysis is then completely different and the measure $m$ cannot be expected to remain a density in general. 

In fact part of our analysis is close to the one developed in Cardaliaguet, Carlier and Nazaret  \cite{CCN} for a particular optimal transport problem: in particular the notion of weak solution is similar to the one introduced in \cite{CCN}. However, some points of the analysis for the MFG system differ substantially from \cite{CCN}: first the framework, if more regular, is also much more general (with space dependence for the Hamiltonian): a substential amount of the work consists in overcoming this issue. Second,  \cite{CCN} strongly benefited from the $L^\infty-$estimate on $m$ obtained for optimal transport problems by Carrillo, Lisini, Savar\'e and D. Slepcev \cite{clss}. Here we are not aware of such bound in general. On another hand, estimates for solutions of HJ equations with unbounded right-hand side,  obtained in Cardaliaguet and Silvestre \cite{CS}, provide Hölder bounds on  $\phi$: this allows to overcome the difficulty of unbounded $m$. Finally, the uniqueness arguments for $\phi$ differ from that of \cite{CCN}, where the optimal control problem of the HJ equation was more ``strictly convex". 

Our second contribution is to apply the above MFG system to differential games with a finite number of players: we show that  the ``optimal strategies" in the Hamilton-Jacobi part of \eqref{MFG} can be used to derive approximate Nash equilibria for differential games with finitely many players. The idea is that, when the number of players is large, it is useless to look at the other player's behavior: one just needs to use the open-loop strategy given by the mean field game system. This idea is reminiscent of several results obtained for second order MFG systems with nonlocal coupling \cite{CD2012}, \cite{HCMieeeAC06}. Compared to these works we face here two difficulties: the first one is the lack of regularity of our weak solutions of the MFG system \eqref{MFG}. The second is that the  local nature of our coupling: this obliges us to be very careful  in defining the payoff of the differential game with finitely many  players. 

We complete the paper by the remark that, if  $(m,\phi)$ is the weak solution of the MFG system \eqref{MFG}, then $\phi$ is a viscosity solution of a second order elliptic equation in time and space. We   thus recover a result explained in \cite{LLperso} in a more regular framework. The difference with   \cite{LLperso} is that we have here to carefully handle the points  where $m$ vanishes: our equation becomes a variational inequality instead of a pure quasi-linear elliptic equation as in  \cite{LLperso}. 

The paper is organized as follows: we first introduce the notations and assumptions used all along the paper (section \ref{sec:notass}). Then we introduce the two optimal control problems, one of HJ equation, the other of continuity equation; we prove that these two problems are in duality (section \ref{sec:2opti}). The main issue is to show that the first problem has a solution: this is the aim of section \ref{sec:optiHJ}. Then we are ready to define the notion of weak solution for \eqref{MFG} and to show that the system has a unique solution (section \ref{sec:exuniq}). We complete the paper by showing how to use the solution of the mean field game system to finitely many player differential games (section \ref{sec:jeufini}). Following \cite{LLperso}, we also show that, if $(m,\phi)$ is the solution to \eqref{MFG}, then the map $\phi$ is also a viscosity solution of a second order equation in time space (section \ref{sec:visco}).   \\

{\bf Acknowledgement: }  
This work has been partially supported by the Commission of the
European Communities under the 7-th Framework Programme Marie
Curie Initial Training Networks   Project SADCO,
FP7-PEOPLE-2010-ITN, No 264735, and by the French National Research Agency
 ANR-10-BLAN 0112 and ANR-12-BS01-0008-01.

\section{Notations and assumptions}\label{sec:notass}

\noindent {\bf Notations :} We denote by 
 $\lg x, y\rg$ the Euclidean scalar product  of two vectors $x,y\in\R^d$ and by
$|x|$ the Euclidean norm of $x$. For any $x_0\in\R^d$ and $r>0$, we denote by $B_r(x_0)$ the open ball of radius $r$, centered at $x_0\in\R^d$, and we set $B_r=B_r(0)$. We use a similar notation $B_r(t_0,x_0)$ for a ball  of $\R^{d+1}$ centered at $(t_0,x_0)\in \R\times \R^d$ and of radius $r>0$.  

We work in the flat $d-$dimensional torus $\T^d=\R^d\backslash \Z^d$. We denote by $P(\T^d)$ the set of Borel probability measures over $\T^d$. It is endowed with the weak convergence. For $k,n\in\N$ and $T>0$, we denote by ${\mathcal C}^k([0,T]\times \T^d, \R^n)$ the space of maps $\phi=\phi(t,x)$ of class ${\mathcal C}^k$ in time and space with values in $\R^n$. For $p\in [1,\infty]$ and $T>0$, we denote by $L^p(\T^d)$ and $L^p((0,T)\times \T^d)$ the set of $p-$integrable maps over $\T^d$ and $[0,T]\times \T^d$ respectively. We often abbreviate $L^p(\T^d)$ and $L^p((0,T)\times \T^d)$ into  $L^p$. We denote by $\|f\|_p$ the $L^p-$norm of a map $f\in L^p$. 

If $\mu$ is a vector measure over $\T^d$ or $[0,T]\times \T^d$, we denote by $\mu^{ac}$ and $\mu^s$ the decomposition of $\mu$ in absolutely continuous part and singular part with respect to the Lebesgue measure. Recall that $\mu=\mu^{ac}+\mu^s$. For simplicity, if $\phi\in BV$ over $[0,T]\times \T^d$,  we abbreviate the notation $(\partial_t\phi)^{ac}$ and $(\partial_t\phi)^{s}$ into  $\partial_t\phi^{ac}$ and  $\partial_t\phi^{s}$ respectively. \\

%
%
%
%
%
%
%
%
%

\noindent {\bf Assumptions:} We now collect the assumptions on the coupling $f$, the Hamiltonian $H$ and the initial and terminal conditions $m_0$ and $\phi_T$. These conditions are supposed to hold throughout the paper. 
\begin{itemize}
\item[(H1)] (Condition on the coupling) the coupling $f:\T^d\times [0,+\infty)\to \R$ is continuous in both variables, increasing with respect to the second variable $m$, 
and there exist $q>1$  and $\Cl[C]{C0f}$  such that 
\be\label{Hypf}
 \frac{1}{\Cr{C0f}}|m|^{q-1}-\Cr{C0f}\leq f(x,m) \leq \Cr{C0f} |m|^{q-1}+\Cr{C0f} \qquad \forall m\geq 0 \;.
\ee
Moreover we ask the following normalization condition to hold:
\be\label{Hypf(0,m)=0}
f(x,0)=0 \qquad \forall x\in \T^d\;.
\ee

\item[(H2)] (Conditions on the Hamiltonian) The Hamiltonian  $H:\T^d\times \R^d\to\R$ is continuous in both variables, convex and differentiable in the second variable, with $D_pH$ continuous in both variable, and has a superlinear growth in the gradient variable: there exist $r>0$ and  $\Cl[C]{C0H}>0$ such that
\be\label{condr}
r> d(q-1)\wedge 1
\ee
and
\be\label{HypGrowthH}
\frac{1}{r\Cr{C0H}} |\xi|^{r} -\Cr{C0H}\leq H(x,\xi) \leq \frac{\Cr{C0H}}{r}|\xi|^r+\Cr{C0H}\qquad \forall (x,\xi)\in \T^d\times \R^d\;.
\ee
We note for later use that the Fenchel conjugate $H^*$ of $H$ with respect to the second variable is continuous and satisfies similar inequalities
\be\label{HypHstar}
\frac{1}{r'\Cr{C0H}} |\xi|^{r'} -\Cr{C0H}\leq H^*(x,\xi) \leq \frac{\Cr{C0H}}{r'}|\xi|^{r'}+\Cr{C0H}\qquad \forall (x,\xi)\in \T^d\times \R^d\;,
\ee
where $r'$ is the conjugate of $r$: $\ds\frac{1}{r}+\frac{1}{r'}=1$. 

\item[(H3)] (Dependence of $H$ with respect to $x$) We also assume that there is $\theta\in [0,\frac{r}{d+1})$ and a constant $\Cl[C]{Regu}>0$ such that 
\be\label{HypHx}
|H(x,\xi)-H(y,\xi)|\leq \Cr{Regu} |x-y|  \left(|\xi|\vee 1\right)^\theta\qquad \forall x,y\in \T^d,\ \xi \in \R^d
\ee

\item[(H4)] (Conditions on the initial and terminal conditions) $\phi_T:\T^d\to \R$ is of class ${\mathcal C}^1$, while $m_0:\T^d\to \R$ is a continuous, with $m_0\geq 0$ and $\ds \int_{\T^d} m_0dx=1$. 
\end{itemize} 


We now comment upon these assumptions. 

Condition (H1), imposing  $f$ to be nondecreasing with respect to the second variable, is natural in the context of mean field game systems. Indeed this assumption is almost necessary for the well-posedness of \eqref{MFG} (see the discussion in \cite{LL07mf}). The growth condition \eqref{Hypf}, on another hand, is less standard: the variational method used in the paper requires it, but it is clearly not necessary for the existence of a solution to \eqref{MFG}. In particular, the lower bound of \eqref{Hypf} prevents a coupling of the form $f(x,m)=\ln(m)$,  which is typically the case in which one expects to have smooth solutions (see the discussion in \cite{LLperso}). 
 
Condition \eqref{Hypf(0,m)=0} is just a normalization condition, which we assume to fix the ideas. This is mostly without loss of generality.  
Indeed, if $f(x,0)$ is Lipschitz continuous and if all the condition (H1)$\dots$(H4) but \eqref{Hypf(0,m)=0} hold, then one just needs to replace $f(x,m)$ by $f(x,m)-f(x,0)$ and $H(x,p)$ by $H(x,p)-f(x,0)$: the new $H$ and $f$ still satisfy the above conditions (H1)$\dots$(H4) with
\eqref{Hypf(0,m)=0}. 

Let us set 
$$\ds F(x,m)=
\left\{\begin{array}{ll}
\ds \int_0^m f(x,\tau)d\tau & {\rm if }\; m\geq 0\\
+\infty & {\rm otherwise}
\end{array}\right. 
$$
Then is convex in $m$. It is continuous on $\T^d\times (0,+\infty)$, derivable and strictly convex in $m$  and satisfies 
\be\label{HypGrowthF}
 \frac{1}{q\Cr{C0f}}|m|^{q}-\Cr{C0f}\leq F(x,m) \leq  \frac{\Cr{C0f}}{q}|m|^{q}+\Cr{C0f} \qquad \forall m\geq 0
\ee
(changing the constant $\Cr{C0f}$ if necessary). 
Let $F^*$ be the convex conjugate of $F$ with respect to the second variable. Note that $F^*(x,a)=0$ for $a\leq 0$ because $F(x,m)$ is nonnegative and equal to $+\infty$ for $m<0$.  Moreover, 
\be\label{HypGrowthFstar}
\frac{1}{p\Cr{C0f}}|a|^{p}-\Cr{C0f}\leq F^*(x,a) \leq  \frac{\Cr{C0f}}{p}|a|^{p}+\Cr{C0f} \qquad \forall a\geq 0\;,
\ee
where $p$ is the conjugate of $q$: $1/p+1/q=1$. 

Assuming, as in (H2),  that $H$ has a superlinear growth is rather natural for HJ equations: this condition is known to ensure ``Lipschitz bounds" on the solutions of the associated HJ equation.  However, as the right-hand side of \eqref{MFG}-(i) is time dependent, this Lipschitz bound is lost in general and has to be replaced by Hölder estimates, at least when the right-hand side of \eqref{MFG}-(i) is bounded (cf. \cite{CC}, \cite{CaRa}). In our context we face the additional difficulty that, instead of $L^\infty$ bound on the right-hand side of \eqref{MFG}-(i), we only know that it is bounded in $L^p$. The condition that the growth rate $r$ is larger than $d(q-1)$ is precisely used to handle this issue: indeed it allows  to prove that solutions of \eqref{MFG}-(i) are bounded even when the right-hand side is in $L^p$. Moreover, this assumption plays a key role to guaranty the Hölder regularity of solutions of HJ with such unbounded right-hand side: this has been established---in a much broader context---in \cite{CS}, and we recall the estimate in Lemma \ref{CShold} below. 

Finally, some comment upon assumption (H3) are in order. This technical assumption induces a strong restriction upon the dependence of the leading term of $H$ (i.e., the term of order $|p|^r$) with respect to $x$. For instance, if $\ds H(x,\xi)= |\xi|^r -\ell(x)$, assumption (H3) reduces to $\ell$ Lipschitz continuous (since we can take $\theta=0$). On another hand, the condition excludes Hamiltonians of the form $H(x,\xi)= a(x)|\xi|^r$, with $\frac{1}{\Cr{Regu}} \leq a(x)\leq \Cr{Regu}$, because here $\theta= r\notin [0, \frac{r}{d+1})$.  We have to require \eqref{HypHx} to show that the regularization by (classical) convolution of subsolution of \eqref{MFG}-(i) is still a subsolution with a controlled error (see the proof of Proposition \ref{prop:valeursegales}): it is not clear that this computation is optimal.


Throughout the paper, we will have to deal with Hamilton-Jacobi equations in which the right-hand side is unbounded. To handle the solutions of these equations, the following result will be useful: 

\begin{Lemma} [\cite{CS}, Theorem 1.3] \label{CShold} Let $H$ satisfies (H2),  $p>1+ d/r$ and $\alpha\in C^0((0,T)\times \T^d)$. Then, any 
continuous  viscosity solution $u$  of 
\be\label{eqHJ}
- \partial_t \phi +H(x,D\phi) =\alpha(t,x) \qquad {\rm in }\; (0,T)\times \T^d
\ee
is of class $C^\beta$ in any compact subset $K$ of $[0,T)\times \T^d$, where $\beta$ ($\in (0,1)$) and the $C^\beta$ norm depend on the compact $K$, on $\|u\|_\infty$, on $p$, $d$, $r$ and $\Cr{C0H}$, and on $\|(\alpha)_-\|_\infty$ and $\|\alpha\|_p$. 
\end{Lemma}

In other words, if $\alpha$ is bounded below and in $L^p$, then we have a control on the Hölder norm of the solution $u$ of \eqref{eqHJ}. Note that in \cite{CS} the result is given for (possibly degenerate) second order parabolic equations under the condition $r>2$: a careful inspection of the proof shows that for first order HJ equations, the assumption $r>1$ suffices for the estimate. 

\section{Two optimization problems}\label{sec:2opti}

 The aim of this section is to introduce two optimization problems and show that these problems are in duality. Throughout this section, the maps $\phi_T$ and $m_0$ are fixed and are assumed to satisfy condition (H4).

\subsection{The optimal control of Hamilton-Jacobi equations}\label{subsecOptiHJ}

To describe the first optimization problem,  let us denote by $\mathcal K_0$ the set of   maps $\phi\in {\mathcal C}^1([0,T]\times \T^d)$ such that $\phi(T,x)= \phi_T(x)$ and define, on $\mathcal K_0$, the functional 
\be\label{DefmathcalA}
{\mathcal A}(\phi)= \int_0^T\int_{\T^d}  F^*\left(x,-\partial_t\phi(t,x)+H(x,D\phi(t,x)) \right)\ dxdt - \int_{\T^d} \phi(0,x)dm_0(x). 
\ee
Our first optimization problem is 
\be\label{PB:dual2}
\inf_{\phi\in \mathcal K_0} \mathcal A(\phi)
\ee
In order to give an interpretation of problem \eqref{PB:dual2}, let us set, for $\phi\in \mathcal K_0$, $$
\alpha(t,x)= -\partial_t\phi(t,x)+H(x,D\phi(t,x)).$$ Then we can see $\alpha$ as a control which, combined with the terminal condition $\phi(T,\cdot)= \phi_T$, determines $\phi$ as a solution of an HJ equation. Minimizing ${\mathcal A}$ can be interpreted as  an optimal control problem for the Hamilton-Jacobi equation
$$
\left\{ \begin{array}{l}
 -\partial_t\phi(t,x)+H(x,D\phi(t,x))= \alpha(t,x)\qquad {\rm in}\; (0,T)\times \T^d\\
 \phi(T,x)=\phi_T(x)\qquad {\rm in}\;  \T^d
\end{array}\right.
$$
for the criterium 
$$
\int_0^T\int_{\T^d} F^*\left(x,\alpha(t,x) \right)\ dxdt - \int_{\T^d} \phi(0,x)dm_0(x). 
$$

\subsection{The optimal control of the continuity equation}

To describe the second optimization problem, let us denote by $\mathcal K_1$ the set of pairs $(m, w)\in L^1((0,T)\times \T^d) \times L^1((0,T)\times \T^d,\R^d)$ such that $m(t,x)\geq 0$ a.e., with $\ds \int_{\T^d}m(t,x)dx=1$ for a.e. $t\in (0,T)$, and which satisfy in the sense of distributions the continuity equation
\be\label{conteq}
\partial_t m+{\rm div} (w)=0\; {\rm in}\; (0,T)\times \T^d, \qquad m(0)=m_0. 
\ee
We define on $\mathcal K_1$ the functional
$$
{\mathcal B}(m,w)=  \int_0^T\int_{\T^d} m(t,x) H^*\left(x, -\frac{w(t,x)}{m(t,x)}\right)+ F(x,m(t,x)) \ dxdt + \int_{\T^d} \phi_T(x)m(T,x)dx.
$$
Let us first give a precise meaning to ${\mathcal B}$. If $m(t,x)=0$, then  by convention
$$
m H^*\left(x, -\frac{w}{m}\right)=\left\{\begin{array}{ll}
+\infty & {\rm if }\; w\neq 0\\
0 & {\rm if }\; w=0
\end{array}\right.
$$
As $H^*$ and $F$ are bounded below and $m\geq 0$ a.e., the first integral in ${\mathcal B}(m,w)$  is well defined in  $\R\cup\{+\infty\}$. The term $\ \int_{\T^d} \phi_T(x)m(T,x)dx\ $ has to be interpreted as follows: let us set $\ \ds v(t,x)= -\frac{w(t,x)}{m(t,x)}\ $ if $m(t,x)>0$ and $\ v(t,x)=0\ $ otherwise. Because of the growth of $H^*$ (thanks to \eqref{HypHstar}, which is a consequence of (H2)), ${\mathcal B}(m,w)$ is infinite if $v\notin L^{r'}(m\ dxdt)$. So we can assume without loss of generality that  $v\in L^{r'}(m\ dxdt)$. In this case equation \eqref{conteq} can be rewritten as the continuity equation 
\be\label{conteq12}
\partial_t m+{\rm div} (mv)=0\; {\rm in}\; (0,T)\times \T^d, \qquad m(0)=m_0. 
\ee
As $v\in L^{r'}(m\ dxdt)$, it is well-known that $m$ can be identified with a continuous map from $[0,T]$ to $P(\T^d)$ (see, e.g., \cite{AGS}). In particular, the measure $m(t)$ is defined for any $t$, which gives a meaning to the second integral term in the definition of ${\mathcal B}(m,w)$. 

The second optimal control problem is the following: 
\be\label{Pb:mw2}
\inf_{(m,w)\in \mathcal K_1} \mathcal B(m,w)\;.
\ee
The introduction of $v$ gives a natural interpretation of \eqref{Pb:mw2}: indeed one can see the vector field $v$ as a control over the state $m$ through the continuity equation \eqref{conteq12}. In this case the optimization of ${\mathcal B}$ can be viewed as an optimal control of \eqref{conteq12}. 

\subsection{The two problems are in duality}

\begin{Lemma}\label{Lem:dualite} We have
$$
\inf_{\phi\in \mathcal K_0}{\mathcal A}(\phi) = - \min_{(m,w)\in \mathcal K_1} {\mathcal B}(m,w), 
$$
Moreover, the minimum in the right-hand side is achieved by a unique pair $(m,w)\in \mathcal K_1$ satisfying $(m,w)\in  L^q((0,T)\times \T^d)\times L^{\frac{r'q}{r'+q-1}}((0,T)\times \T^d)$. 
\end{Lemma}

\begin{Remark}{\rm Note that $\frac{r'q}{r'+q-1}>1$ because $r'>1$ and $q>1$. 
}\end{Remark}

\begin{proof} We use the Fenchel-Rockafellar duality theorem (cf. e.g., \cite{ET}). For this, we rewrite the first optimization problem \eqref{PB:dual2} in a more suitable form. 
Let  $E_0= {\mathcal C}^1([0,T]\times \T^d)$ and  $E_1=   {\mathcal C}^0([0,T]\times \T^d, \R)\times {\mathcal C}^0([0,T]\times \T^d, \R^d)$. We define on $E_0$  the functional 
$$
{\mathcal F}(\phi)= -\int_{\T^d} m_0(x)\phi(0,x)dx +\chi_S(\phi),  
$$
where $\chi_S$ is the characteristic function of the set $S=\{\phi\in E_0, \; \phi(T, \cdot)= \phi_T\}$, i.e., $\chi_S(\phi)=0$ if $\phi\in S$ and $+\infty$ otherwise. For $(a,b)\in E_1$, we set 
$$
{\mathcal G} (a,b)= \int_0^T\int_{\T^d}  F^*(x,-a(t,x)+H(x,b(t,x)) )\ dxdt \;.
$$
Note that ${\mathcal F}$ is convex and lower semi-continuous on $E_0$ while ${\mathcal G}$ is convex  and continuous on $E_1$. Let $\Lambda:E_0\to E_1$ be the bounded linear operator defined by $\Lambda (\phi)= (\partial_t\phi, D\phi)$. Note that 
$$
\inf_{\phi\in \mathcal K_0} \mathcal A(\phi)= \inf_{\phi\in E_0} \left\{ {\mathcal F}(\phi)+{\mathcal G}(\Lambda(\phi))\right\}.
$$
One easily checks that there is a map $\phi$ such that ${\mathcal F}(\phi)<+\infty$ and 
such that ${\mathcal G} $ is continuous at $\Lambda (\phi)$: just take $\phi(t,x)=\phi_T(x)$. 


By the Fenchel-Rockafellar duality theorem we have 
$$
\inf_{\phi\in E_0} \left\{ {\mathcal F}(\phi)+{\mathcal G}(\Lambda(\phi))\right\}
=
\max_{ (m,w)\in E_1'} \left\{ -  {\mathcal F}^*(\Lambda^*(m,w))-{\mathcal G}^*(-(m,w))\right\}
$$
where $E_1'$ is the dual space of $E_1$, i.e., the set of vector valued Radon measures $(m, w)$ over $[0,T]\times \T^d$ with values in $\R\times \R^d$ and ${\mathcal F}^*$ and ${\mathcal G}^*$ are the convex conjugates of ${\mathcal F}$ and ${\mathcal G}$ respectively. By a direct computation we have 
$$
{\mathcal F}^*(\Lambda^*(m,w))
= \left\{ \begin{array}{ll}
\ds \int_{\T^d} \phi_T(x)dm(T,x) & {\rm if }\; \partial_t m+{\rm div} (w)=0, \; m(0)=m_0\\
+\infty & {\rm otherwise}
\end{array}\right.
$$
where the equation $\ \partial_t m+{\rm div} (w)=0, \; m(0)=m_0\ $ holds in the sense of distribution. 
Let us  set 
$$
K(x,a,b)=  F^*(x,-a+H(x,b) )\qquad \forall (x,a,b)\in \T^d\times \R\times \R^d\;.
$$
Then, for any $(m,w)\in \R\times \R^d$, 
$$
\begin{array}{rl}
\ds K^*(x,m,w)\; = & \ds \sup_{(a,b)\in \R\times \R^d} \left\{am+\lg b, w\rg -  F^*(x,-a+H(x,b) ) \right\}\\ 
= & \ds  \sup_{(a,b)\in \R\times \R^d} \left\{H(x,b)m-am+\lg b, w\rg -  F^*(x,a ) \right\}\\ 
= & \ds \sup_{b\in \R^d} \left\{ H(x,b)m + \lg b, w\rg +  F(x,-m ) \right\}
\end{array}
$$
Since $H$ is convex with respect to the second variable and has a superlinear growth, we have therefore 
$$
K^*(x,m,w)= \left\{ \begin{array}{ll}
\ds   F(x,-m ) -m H^*( x, -\frac{w}{m}) & {\rm if }\; m<0 \\ 
 0 & {\rm if }\;  m=0, \ w=0\\
+\infty & {\rm otherwise}
\end{array}\right.
$$
In particular, since, from \eqref{HypHstar} and \eqref{HypGrowthF},  $H^*$ has a superlinear growth and $F$ is coercive, the recession function  $K^{*\infty}$ of $K^*$ satisfies:
$$
K^{*\infty}(x,m,w)= 
\left\{ \begin{array}{ll}
\ds   0 & {\rm if }\;   m=0, \ w=0\\
+\infty & {\rm otherwise}
\end{array}\right.
$$
Therefore ${\mathcal G}^*(m,w)=+\infty$ if $(m,w)\notin L^1$ and, if $(m,w)\in L^1$, 
$$
{\mathcal G}^*(m,w)= \int_0^T\int_{\T^d} K^*(x,m(t,x),w(t,x))dtdx .
$$
Accordingly 
$$
\begin{array}{l}
\ds \max_{ (m,w)\in E_1'} \left\{ -  {\mathcal F}^*(\Lambda^*(m,w))-{\mathcal G}(m,w)\right\}\\
\qquad \qquad \qquad \ds = \max \left\{ \int_0^T\int_{\T^d}  -F(x,m ) -m H^*( x, -\frac{w}{m})\ dtdx  -\int_{\T^d} \phi_T(x) m(T,x)\ dx \right\}
 \end{array}
$$
where the maximum is taken over the $L^1$ maps $(m,w)$ such that $m\geq 0$ a.e. and 
$$
\partial_t m+{\rm div} (w)=0, \; m(0)=m_0.
$$
As $\ds\ \int_{\T^d} m_0=1\ $, we have therefore $\ds\ \int_{\T^d} m(t)=1\ $ for any $t\in [0,T]$. Thus the pair $(m,w)$ belongs to the set $\mathcal K_1$.

Let now $(m, w)\in \mathcal K_1$ be optimal in the above system. From the growth conditions \eqref{HypGrowthH} and \eqref{HypGrowthF}, we have 
$$
\begin{array}{rl}
 C \; \geq & \ds  \int_0^T\int_{\T^d}  F(x,m ) +m H^*( x, -\frac{w}{m})\ dtdx  +\int_{\T^d} \phi_T(x) m(T,x)\ dx \\
  \geq & \ds   \int_0^T\int_{\T^d}  \left(\frac{1}{ C}|m|^{q}+
 \frac{m}{ C } \left|\frac{w}{m}\right|^{r'} - C \right) dxdt -\|\phi_T\|_\infty 
 \end{array}
 $$
In particular, $m\in L^q$.  By Hölder inequality, we also have 
$$
\int_0^T\int_{\T^d} |w|^{\frac{r'q}{r'+q-1}} = \int\int_{\{m>0\}} |w|^{\frac{r'q}{r'+q-1}} \leq \|m\|_q^{\frac{r'-1}{r'+q-1}} \left(\int\int_{\{m>0\}} \frac{|w|^{r'}}{m^{r'-1}} \right)^{\frac{q}{r'+q-1}} \leq C
$$
so that $w\in L^{\frac{r'q}{r'+q-1}}$. Finally, we note that there is a unique minimizer to \eqref{Pb:mw2}, because the set $\mathcal K_1$ is convex and the maps $F(x,\cdot)$ and $H^*(x,\cdot)$  are strictly convex: thus $m$ is unique and so is $\ds \frac{w}{m}$ in $\{m>0\}$. As $w=0$ in $\{m=0\}$, uniqueness of $w$ follows as well. 
\end{proof}

\section{Analysis of the optimal control of the HJ equation}\label{sec:optiHJ}

In general, we do not expect problem \eqref{PB:dual2} to have a solution. In this section we exhibit a relaxation for  \eqref{PB:dual2}  (Proposition \ref{prop:valeursegales}) and show that this relaxed problem has at least one solution (Proposition \ref{Prop:existence}).

\subsection{The relaxed problem}

Let ${\mathcal K}$ be the set of pairs $(\phi,\alpha)\in BV((0,T)\times \T^d)\times L^p((0,T)\times \T^d)$ such that $D\phi\in L^r((0,T)\times \T^d)$ and which satisfies $\phi(T,x)= \phi_T(x)$ (in the sense of traces) and, 
in the sense of distribution,
\be\label{ineq:phi}
-\partial_t \phi+H(x,D\phi) \leq \alpha \qquad {\rm in } \; (0,T)\times \T^d\;.
\ee
Note that ${\mathcal K}$ is a convex set and that the set $\mathcal K_0$ (defined in Subsection \ref{subsecOptiHJ}) can naturally be embeded into $\mathcal K$: indeed, if $\phi\in {\mathcal K}_0$, then the pair $(\phi,-\partial_t\phi+H(x,D\phi))$ belongs to ${\mathcal K}$. We extend to $\mathcal K$ the functional $\mathcal A$ defined on $\mathcal K_0$ by setting (with a slight abuse of notation)
$$
\mathcal A(\phi,\alpha) = \int_0^T\int_{\T^d}  F^*(x,\alpha(x,t))\ dxdt - \int_{\T^d} \phi(x,0)m_0(x)\ dx\qquad \forall (\phi, \alpha)\in {\mathcal K}.
$$
The next Proposition explains that the problem
\be\label{PB:dual-relaxed}
\inf_{(\phi, \alpha)\in {\mathcal K}} \mathcal A(\phi,\alpha)
\ee
is the relaxed problem of \eqref{PB:dual2}. 

\begin{Proposition}\label{prop:valeursegales} We have
$$
 \inf_{\phi\in \mathcal K_0} \mathcal A(\phi)= \inf_{(\phi, \alpha)\in {\mathcal K}} \mathcal A(\phi,\alpha).
 $$
\end{Proposition}

In order to prove Proposition \ref{prop:valeursegales}, we need a remark which is repeatedly used in the sequel. It says that one can restrict the minimization problem to pairs $(\phi,\alpha)$ for which $\alpha$ is nonnegative.

\begin{Lemma}\label{lem:positif} We have 
$$
\inf_{(\phi, \alpha)\in {\mathcal K}} \mathcal A(\phi,\alpha)= \inf_{(\phi, \alpha)\in {\mathcal K}, \ \alpha\geq0\ {\rm a.e.}} \mathcal A(\phi,\alpha)
$$
\end{Lemma}

\begin{proof} For $(\phi, \alpha)\in {\mathcal K}$, let us set $\tilde \alpha = \alpha\vee 0$. Then $(\phi, \tilde \alpha)\in {\mathcal K}$, $\tilde \alpha\geq 0$ a.e.  and 
$$
\begin{array}{rl}
\ds \mathcal A(\phi,\tilde \alpha) \; = & \ds \int_0^T\int_{\T^d}  F^*(x,\tilde \alpha(x,t))\ dxdt - \int_{\T^d} \phi(x,0)m_0(x)\ dx\\
\leq &\ds \int_0^T\int_{\T^d}  F^*(x,\alpha(x,t))\ dxdt - \int_{\T^d} \phi(x,0)m_0(x)\ dx =\mathcal A(\phi, \alpha)
\end{array}
$$
where the inequality holds because $0$ is a global minimum of $F^*(x,\cdot)$ for any $x\in \T^d$. 
\end{proof}

\begin{proof}[Proof of Proposition \ref{prop:valeursegales}]
 Inequality $\ \ds  \inf_{\phi\in \mathcal K_0} \mathcal A(\phi)\geq \inf_{(\phi, \alpha)\in {\mathcal K}} \mathcal A(\phi,\alpha)\ $ being obvious, let us check the reverse one. Let $(\phi,\alpha)\in {\mathcal K}$. From Lemma \ref{lem:positif} we can assume with loss of generality that $ \alpha\geq 0$ a.e..
 Fix $\ep>0$. Let us first slightly translate and extend $(\phi,\alpha)$ to the larger interval $[-\ep,T+\ep]$: we set 
 $$
 \tilde \phi(t,x)= \left\{\begin{array}{ll} \phi(t+2\ep, x) & {\rm if }\; t\in [-2\ep,T-2\ep)\\
 \phi_T(x)+ \lambda (T-2\ep-t) & {\rm if }\; t\in [T-2\ep,T+2\ep]
 \end{array}\right.
 $$
where $\lambda = -\max_x H(x,D\phi_T(x))$ and 
 $$
 \tilde \alpha(t,x)= \left\{\begin{array}{ll} \alpha(t+2\ep, x) & {\rm if }\; t\in [-2\ep,T-2\ep)\\
 0 & {\rm if }\; t\in [T-2\ep,T+2\ep]
 \end{array}\right.
 $$
 One easily checks that $(\tilde \phi,\tilde \alpha)$ satisfies in the sense of distribution $-\partial_t \tilde \phi+H(x,D\tilde \phi)\leq \tilde \alpha$ in 
 $(-2\ep, T+2\ep)\times \T^d$.

  We regularize $(\tilde \phi,\tilde \alpha)$ by convolution: 
let $\xi$ be a smooth convolution kernel in $\R^{d+1}$ with support in the unit ball, with $\xi\geq 0$ and $\int\xi=1$. Let us set $\xi_\ep(t,x)= \ep^{-d-1} \xi((t,x)/\ep)$ and $\phi_\ep=\xi_\ep \star \tilde \phi$. Then we have
$$
-\partial_t \phi_\ep+\xi_\ep\star H(\cdot,D\phi_\ep) \leq \xi_\ep \star \tilde \alpha \qquad {\rm in } \; (0,T)\times \T^d\;.
$$
By convexity of $H$ with respect to the second variable, we have
$$
H(x,D\phi_\ep(t,x)) \leq (\xi_\ep\star H(\cdot,D\phi))(t,x) + \beta_\ep(t,x)
$$
where 
$$
\beta_\ep(t,x) =\int_{B_\ep(t,x)}  \xi_\ep((t,x)-(s,y))\left| H(y, D\phi(s,y))-H(x,D\phi(s,y))\right| \ dsdy
$$
In view of assumption \eqref{HypHx}, we have, setting $\delta= r/\theta$ ($>1$) and $\delta'= \delta/(\delta-1)$ and using Hölder inequality, 
$$
\begin{array}{rl}
\beta_\ep(t,x) \; \leq & \ds \int_{B_\ep(t,x)}  \xi_\ep((t,x)-(s,y))|y-x|(1\vee | D\phi(s,y))|)^{\theta} \ dsdy\\ 
\leq & \ds C\ep  \left(\int_{B_\ep(t,x)} \xi_\ep^{\delta'}((t,x)-(s,y))\right)^{1/\delta'} \left(\int_{B_\ep(t,x)} (1\vee | D\phi(s,y))|)^{r}\right)^{1/\delta} \\ 
\leq & \ds C  \ep^{1-(d+1)\theta/r} (1+ \|D\phi\|_r^\theta)
\end{array}
$$
Recall that, by assumption (H2), $\theta< r/(d+1)$, so that  $1-(d+1)\theta/r>0$. Let us set
$$
\alpha_\ep=\xi_\ep \star \tilde \alpha+C  \ep^{1-(d+1)\theta/r}(1+ \|D\phi\|_r^\theta).
$$
The previous estimates show that the pair $(\phi_\ep, \alpha_\ep)$ satisfies 
\be\label{phiepalphaep}
-\partial_t \phi_\ep+H(x,D\phi_\ep) \leq \alpha_\ep \qquad {\rm in } \; (0,T)\times \T^d\;.
\ee
In order to fulfill the terminal condition $\phi_\ep(T,\cdot)=\phi_T$, we must once more slightly  modify $\phi_\ep$. For this we note that, by regularity of $\phi_T$ we have 
\be\label{phiprochephiT}
\sup_{t\in [T-\ep,T]} \|\phi_\ep(t, \cdot)-\phi_T\|_\infty\leq C\ep\;.
\ee
Let $\zeta_\ep:\R\to \R$ be a smooth, nondecreasing map, with $\zeta_\ep=0$ in $(-\infty, T-\ep]$ and $\zeta_\ep=1$ in $[T,+\infty)$ and such that $\|\zeta_\ep'\|_\infty\leq C\ep^{-1}$. We set 
$$
\tilde \phi_\ep (t,x)= (1-\zeta_\ep(t)) \phi_\ep(t,x)+\zeta_\ep(t)(\phi_T(x)+\lambda(T-t))\qquad \forall (t,x)\in [0,T]\times \T^d\;.
$$
Then $\tilde \phi_\ep(T,\cdot)=\phi_T$ and, by using the convexity of $H$, estimates \eqref{phiepalphaep} and \eqref{phiprochephiT} as well as the definition of $\zeta_\ep$, we get
$$
-\partial_t \tilde\phi_\ep +H(x,D\tilde\phi_\ep) \leq \left\{\begin{array}{ll}
\alpha_\ep & {\rm in} \; (0,T-\ep)\times \T^d\\
\alpha_\ep + C & {\rm in} \; (T-\ep,T)\times \T^d
\end{array}\right. 
$$
Therefore 
\be\label{inPbdualleq}
\begin{array}{rl}
\ds  \inf_{\psi\in \mathcal K_0} \mathcal A(\psi) \;  \leq & \ds 
\int_0^T\int_{\T^d}  F^*(x,-\partial_t\tilde \phi_\ep+H(x,D\tilde\phi_\ep) )\ dxdt - \int_{\T^d} \tilde\phi_\ep(x,0)dm_0(x) \\
\leq & \ds \int_0^{T-\ep}\int_{\T^d}  F^*(x,\alpha_\ep )\ dxdt +\int_{T-\ep}^T\int_{\T^d}  F^*(x,\alpha_\ep + C )\ dxdt
- \int_{\T^d} \tilde\phi_\ep(0,x)dm_0(x)
\end{array}
\ee
We now let $\ep\to 0$. As $\alpha_\ep\to \alpha$ in $L^p$ while $H^*$ satisfies the growth condition \eqref{HypHstar}, we have 
\be\label{limF*}
\limsup_{\ep\to 0} \int_0^{T-\ep}\int_{\T^d}  F^*(x,\alpha_\ep )\ dxdt\leq 
\int_0^T\int_{\T^d}  F^*(x,\alpha )\ dxdt\;.
\ee
In the same way, 
$$
\limsup_{\ep\to 0} \int_{T-\ep}^T\int_{\T^d}  F^*(x,\alpha_\ep + C )\ dxdt = 0.
$$
 In order to understand the convergence of the term $\ds \int_{\T^d} \tilde\phi_\ep(0,x)dm_0(x)$, we need the following Lemma, in which,  for $t_1<t_1$, we denote by $w(t_1^+,\cdot)$ and $w(t_2^-,\cdot)$ the traces, on the sets $t=t_1$ and $t=t_2$,  of a BV function $w$ restricted to $(t_1,t_2)\times \T^d$.  

\begin{Lemma}\label{lem:estiPhiT} There is a constant $\Cl[C]{estiPhiT}$ such that, for any $(\phi,\alpha)\in {\mathcal K}$ and for any $0\leq t_1< t_2\leq T$, we have
\be\label{eq:lem:estiPhiT}
\phi(t_1^+,\cdot)\leq \phi(t_2^-, \cdot) + \Cr{estiPhiT}(t_2-t_1)^\nu\|\alpha\|_p\qquad {\rm a.e.}
\ee
where 
\be\label{defnu}
\nu := \frac{r-d(q-1)}{ d(q-1)(r-1)+rq}
\ee
(recall that $r-d(q-1)>0$ by assumption (H2), so that $\nu>0$). 
\end{Lemma} 

Admitting for a while the above result, we complete the proof of Proposition \ref{prop:valeursegales}. In view of Lemma \ref{lem:estiPhiT}, we have 
$$
\tilde \phi_\ep (0^+,\cdot) \geq \xi_\ep\star \phi(0,\cdot)-C\ep^\nu\|\alpha\|_p\;.
$$
Hence 
$$
 \liminf_{\ep\to0} \int_{\T^d} \tilde\phi_\ep(0,x)dm_0(x)\geq \int_{\T^d} \phi(0^+,x)dm_0(x)  \;, 
$$
which, combined with \eqref{inPbdualleq} and \eqref{limF*}, shows that 
$$
 \inf_{\psi\in \mathcal K_0} \mathcal A(\psi) \;  \leq\;  \int_0^T\int_{\T^d}  F^*(x,\alpha )\ dxdt - \int_{\T^d} \phi(0,x)dm_0(x)= \mathcal A(\phi,w).
$$
Taking the infimum over $(\phi,w)\in \mathcal K$ gives the result. \end{proof}

\begin{proof}[Proof of Lemma \ref{lem:estiPhiT}]  Let us first assume that $\phi$ and $\alpha$ are of class ${\mathcal C}^1$. Since $\ds r> d(q-1)$, we can also fix $\beta\in (1/r,  \frac{1}{d(q-1)})$. Let $x\in \T^d$ and $0\leq t_1<t_2\leq T$. For any $\sigma\in\R^d$ with $|\sigma|\leq 1$, let us define the arc
$$
x_\sigma(s)= \left\{\begin{array}{ll}
x+ \sigma (s-t_1)^\beta & {\rm if} \; s\in [t_1, \frac{t_1+t_2}{2}]\\
x+\sigma(t_2-s)^\beta & {\rm if} \;  s\in [\frac{t_1+t_2}{2},t_2]  
\end{array}\right.
$$
Let $L$ be the convex conjugate of $p\to H(x,-p)$, i.e., $L(x,\xi)=H^*(x,-\xi)$. Then 
$$
\begin{array}{l}
\ds \frac{d}{ds}\left[ \phi(s,x_\sigma(s))-\int_s^{t_2} L(x_\sigma(\tau), x'_\sigma(\tau))d\tau\right] \\ 
\qquad \qquad = \; \ds \partial_t \phi(s,x_\sigma(s))+ \lg D\phi(s,x_\sigma(s)), 
x'_\sigma(s)\rg +L(x_\sigma(s), x'_\sigma(s))\\ 
\qquad \qquad  \geq \;  \ds  \partial_t \phi(s,x_\sigma(s))- H(x_\sigma(s), D\phi(s,x_\sigma(s))) \;
\geq \;  \ds -\alpha(s,x_\sigma(s))
\end{array}
$$
Integrating first in time on the interval $[{t_1},t_2]$ and then in $\sigma\in B_1$ the above inequality, we get
 $$
\phi({t_1},x)\leq \phi({t_2},x) + \frac{1}{|B_1|}\int_{B_1}\int_{t_1}^{t_2}\left[  L(x_\sigma(s), x'_\sigma(s))+  \alpha(s, x_\sigma(s))\right]dsd\sigma\;.
$$
By assumption \eqref{HypHstar}, we have, on the one hand,
$$
\begin{array}{rl}
\ds \frac{1}{|B_1|}\int_{B_1}\int_{t_1}^{t_2} L(x_\sigma(s), x'_\sigma(s))\ dsd\sigma \;  \leq & \ds \Cr{C0H} \left[ \int_{B_1}\int_{t_1}^{t_2} |x'_\sigma(s))|^{r'}\ dsd\sigma+ ({t_2}-{t_1})\right] \\
\leq & \ds  C( {t_2}-{t_1})^{1-r'(1-\beta)}
\end{array}
 $$
where $1-r'(1-\beta)>0$ since $\beta >1/r$. Using Hölder's inequality, we get, on another hand, 
$$
\begin{array}{rl}
\ds \int_{B_1}\int_{t_1}^{\frac{{t_2}+{t_1}}{2}} \alpha(s, x_\sigma(s))dsd\sigma \;\leq   & \ds \int_{t_1}^{\frac{{t_1}+{t_2}}{2}} \int_{B_1} |\alpha(s, x+ \sigma (s-{t_1})^\beta )| d\sigma ds\\
\leq & \ds  \int_{t_1}^{\frac{{t_1}+{t_2}}{2}} \int_{B(x,  (s-{t_1})^\beta)}  (s-{t_1})^{-d\beta} |\alpha(s, y )| dy ds\\
\leq & \ds  \left[ \int_{t_1}^{\frac{{t_1}+{t_2}}{2}} (s-{t_1})^{-d\beta(q-1)}\right]^{\frac{1}{q}} \|\alpha\|_p\; \leq \; C ({t_2}-{t_1})^{(1-d\beta(q-1))/q} \|\alpha\|_p
\end{array}
$$
where $1-d\beta(q-1)>0$ since $\beta < \frac{1}{d(q-1)}$. In the same way, we have 
$$
 \int_{B_1}\int_{\frac{{t_1}+{t_2}}{2}}^{t_2} \alpha(s, x_\sigma(s))dsd\sigma \;\leq  \; C ({t_2}-{t_1})^{(1-d\beta(q-1))/q} \|\alpha\|_p\;,
 $$
Using the assumption $\ds r> d(q-1)$, one can check that 
$$\ds \beta:= \frac{q(r'-1)+1}{d(q-1)+r'q}= \frac{q+r-1}{d(q-1)(r-1)+rq}$$ satisfies $\beta\in (\frac{1}{r},\frac{1}{d(q-1)}) $. For this choice of $\beta$  we obtain that 
 $$
\phi({t_1},x)\leq \phi(t_2,x) + C ({t_2}-{t_1})^{\frac{r-d(q-1)}{ d(q-1)(r-1)+rq}} \|\alpha\|_p\;.
$$
One gets the result for general  $(\phi,\alpha) \in {\mathcal K}$ by regularizing $(\phi,\alpha)$ by convolution:  let $\xi_\ep$ be as  in the proof of Proposition \ref{prop:valeursegales} and $\phi_\ep=\xi_\ep\star \phi$. Then, by \eqref{phiepalphaep}, we have  
$$
-\partial_t \phi_\ep+H(x,D\phi_\ep) \leq \xi_\ep \star  \alpha +C  \ep^{1-(d+1)\theta/r}(1+ \|D\phi\|_r^\theta) \qquad {\rm in } \; (\ep,T-\ep)\times \T^d\;.
$$
Choose $0<t_1<t_2<T$ such that $\phi_\ep(t_1,\cdot)$ and $\phi_\ep(t_2, \cdot)$ converge a.e. as (a subsequence of) $\ep\to 0$. Using the result in the regular case we have 
 $$
\phi_\ep({t_1},x)\leq \phi_\ep(t_2,x) + C ({t_2}-{t_1})^{\frac{r-d(q-1)}{ d(q-1)(r-1)+rq}} \left(\|\xi_\ep \star  \alpha \|_p+C  \ep^{1-(d+1)\theta/r}(1+ \|D\phi\|_r^\theta)\right) \;.
$$
So
$$
\phi({t_1},\cdot)\leq \phi(t_2,\cdot) + C ({t_2}-{t_1})^{\frac{r-d(q-1)}{ d(q-1)(r-1)+rq}} \|\alpha\|_p\qquad {\rm a.e.}\;.
$$
The above inequality implies \eqref{eq:lem:estiPhiT} because $\phi$ is in BV. 
\end{proof}

\subsection{Existence of a solution for the relaxed problem}

The next proposition explains the interest of considering the relaxed problem \eqref{PB:dual-relaxed} instead of the original one \eqref{PB:dual2}.

\begin{Proposition}\label{Prop:existence} The relaxed problem \eqref{PB:dual-relaxed} has at least one solution $(\phi, \alpha) \in {\mathcal K}$ with the following properties: $\phi$ is continuous on $[0,T]\times \T^d$ and locally Hölder continuous in $[0,T)\times \T^d$ and satisfies in the viscosity sense 
$$
-\partial_t \phi+ H(x,D\phi)\geq 0\qquad {\rm in }\; (0,T)\times \T^d\;.
$$
Moreover, $\ds \alpha = \left(-\partial_t\phi^{ac} + H(x,D\phi)\right)\vee 0$ a.e., where $\partial_t\phi^{ac}$ denotes the absolutely continuous part of the measure $\partial_t\phi$. 
\end{Proposition}


\begin{proof}  Let $(\phi_n)$ be a minimizing sequence for problem \eqref{PB:dual2} and let us set 
$$\alpha_n(t,x)= -\partial_t\phi_n(t,x) +H(x,D\phi_n(t,x)).$$ Regularizing $\phi_n$ if necessary, we can assume without loss of generality that $\phi_n$ is ${\mathcal C}^2$.

According to Lemma \ref{lem:estiPhiT}, we have 
\be\label{ineq:boundphinabove}
\phi_n(t,x)\leq \phi_T(x)+  C(T-t)^\nu\|\alpha_n\|_p, 
\ee
where $\nu>0$ is given by \eqref{defnu}. 
From our growth condition \eqref{HypGrowthFstar} on $F^*$, we have 
$$
\begin{array}{rl}
C \; \geq & \ds \int_0^T\int_{\T^d}  F^*(x,\alpha_n) - \int_{\T^d} \phi_n(0)m_0\\ 
\geq & \ds \frac{1}{p\Cr{C0f}} \int_0^T\int_{\T^d}  \left|\alpha_n\right|^p - \Cr{C0f} T  - \|\phi_T\|_\infty-C\|\alpha_n\|_p
\end{array}
$$
Therefore $(\alpha_n)$ is bounded in $L^p$, so that, from \eqref{ineq:boundphinabove},  $(\phi_n)$ is bounded from above. 

Let now $\psi$ be the viscosity solution to 
$$
\left\{\begin{array}{l}
-\partial_t \psi +H(x,D\psi)= 0\qquad {\rm in }\; (0,T)\times \T^d\\
\psi(T,x)=\phi_T(x)\qquad {\rm in }\;  \T^d
\end{array}\right.
$$
and $\tilde \phi_n$ be the viscosity solution to 
$$
\left\{\begin{array}{l}
-\partial_t \tilde \phi_n +H(x,D\tilde \phi_n)= \alpha_n\vee 0\qquad {\rm in }\; (0,T)\times \T^d\\
\tilde \phi_n(T,x)=\phi_T(x)\qquad {\rm in }\;  \T^d
\end{array}\right.
$$
We set $\tilde \alpha_n=\alpha_n\vee 0$. Then the $\tilde \phi_n$ are Lipschitz continuous and satisfy $-\partial_t \tilde \phi_n +H(x,D\tilde \phi_n)= \tilde \alpha_n$ a.e.. By comparison, we also have $\tilde \phi_n\geq \psi$ and, recalling Lemma \ref{lem:estiPhiT} again, we conclude that $(\tilde \phi_n)$ is uniformly bounded. 

Note also that $(\tilde \phi_n,\tilde \alpha_n)$ is a minimizing sequence for the relaxed problem \eqref{PB:dual-relaxed} because, by comparison,  $\tilde \phi_n\geq \phi_n$ and  $0$ is the minimum of the map $a\to F^*(x,a)$ for any $x$, so that
$$
 \int_0^T\int_{\T^d}  F^*(x,\tilde \alpha_n) - \int_{\T^d} \tilde \phi_n(0)m_0  \leq  \int_0^T\int_{\T^d}  F^*(x, \alpha_n) - \int_{\T^d} \phi_n(0)m_0 \;.
 $$
Using the growth condition \eqref{HypGrowthH} on $H$ we have
$$
\begin{array}{l}
\ds \int_0^T \int_{\T^d} \left( \frac{1}{r\Cr{C0H}} |D\tilde \phi_n |^{r} -\Cr{C0H}\right)\;  \leq \; \ds  \int_0^T \int_{\T^d}  H(x,D\tilde \phi_n) \;
\leq \; \ds \int_0^T \int_{\T^d} \partial_t \tilde \phi_n +  \tilde \alpha_n 
 \; \leq \; C\;,
\end{array}
$$
where the last inequality holds because the $\tilde \phi_n$ are uniformly bounded. 
Accordingly, $(D\tilde \phi_n)$ is bounded in $L^r$ and $(H(x,D\tilde \phi_n) )$ is bounded in $L^1$.  Since $\partial_t \tilde \phi_n =H(x,D\tilde \phi_n)- \tilde \alpha_n$, the sequence $(\partial_t \tilde \phi_n)$  is bounded in $L^1$. This implies that $(\tilde \phi_n)$ is bounded in BV.

Following Lemma \ref{CShold}, we also know that the   $(\tilde \phi_n)$ are uniformly Hölder continuous in any compact subset of $[0,T)\times \T^d$. Accordingly we can assume that $(\tilde \phi_n)$ converge to some $\phi\in BV$ locally uniformly in any compact subset of $[0,T)\times \T^d$, while $(D\tilde \phi_n)$ converges weakly to $D\phi$ in $L^r$ and $(\alpha_n)$ converges weakly to some $\alpha$ in $L^p$. Since $H$ is convex with respect to the last variable, the pair $(\phi,\alpha)$ satisfies 
$-\partial_t\phi+H(x,D\phi)\leq \alpha$ in the sense of distribution. Finally, in view of \eqref{ineq:boundphinabove} and the Lipschitz continuity of $\psi$, we have 
\be\label{boundphiT}
\phi_T(x) - C(T-t) \leq \psi(t,x)\leq \tilde \phi_n(t,x)\leq \phi_T(x) + C(T-t)^\nu,
\ee
so that $\phi(T,x)=\phi_T(x)$ a.e.. 
In particular, $(\phi,\alpha)$ belongs to ${\mathcal K}$. Note also that
$$
 \int_0^T\int_{\T^d}  F^*(x, \alpha) - \int_{\T^d} \phi(0)m_0  \leq \liminf_n  \int_0^T\int_{\T^d}  F^*(x,\tilde \alpha_n) - \int_{\T^d} \tilde \phi_n(0)m_0 .
$$
Therefore the pair $(\phi,\alpha)$ is a minimizer for the relaxed problem \eqref{PB:dual-relaxed}. By construction, $\phi$ is locally Hölder continuous in $[0,T)\times \T^d$ and, by continuity of $\phi_T$ and \eqref{boundphiT}, $\phi$ is also continuous on $[0,T]\times \T^d$ as well. \\

Since, by definition, inequality 
$$
-\partial_t \tilde \phi_n +H(x,D\tilde \phi_n)\geq  0\qquad {\rm in }\; (0,T)\times \T^d
$$
holds in the viscosity sense, we have by passing to the limit that the following inequality holds in the viscosity sense:
$$
-\partial_t  \phi +H(x,D \phi)\geq  0\qquad {\rm in }\; (0,T)\times \T^d\;.
$$

It just remains to prove that $\ds \alpha= \left(H(x,D\phi)-\partial_t\phi^{ac}\right)\vee 0$. 
Recall first that, by construction, $\tilde \alpha_n\geq 0$, so that $\alpha\geq   0$ a.e.. Since $\alpha \in L^p$ and $H(\cdot, D\phi))\in L^1$ and since  the measure $\alpha+\partial_t\phi- H(\cdot,D\phi)$ is nonnegative, its (nonnegative) regular part is given by $\left(\alpha+\partial_t\phi- H(\cdot,D\phi)\right)^{ac}=\alpha+ \partial_t\phi^{ac}- H(\cdot,D\phi)$. Hence $\ds \alpha\geq  \left(H(x,D\phi)-\partial_t\phi^{ac}\right)\vee 0$. Note also that the (nonnegative) singular  part of the measure  $\alpha+\partial_t\phi- H(\cdot,D\phi)$ is given by $\partial_t\phi^s$. Therefore, if we set  $\tilde \alpha:= \left(H(x,D\phi)-\partial_t\phi^{ac}\right)\vee 0$, we have
$$
-\partial_t \phi+ H(x,D\phi)\leq \tilde \alpha \; \leq \; \alpha
$$
in the sense of distribution. Since $a\to F^*(x,a)$ is increasing on $[0,+\infty)$, we have 
$$
 \int_0^T\int_{\T^d}  F^*(x,\tilde \alpha(x,t))\ dxdt
 \leq  \int_0^T\int_{\T^d}  F^*(x,\alpha(x,t))\ dxdt \;.
 $$
By optimality of $(\phi,\alpha)$, this implies that  $F^*(x,\tilde \alpha(x,t))= F^*(x, \alpha(x,t))$ a.e., and therefore that $\tilde \alpha=\alpha$.
\end{proof}

Collecting the arguments used in the proof of Proposition \ref{prop:valeursegales}, Lemma \ref{lem:estiPhiT} and  \ref{Prop:existence} one can show that, given $\alpha\in L^p$,  inequality 
\be\label{subHJ}
\left\{ \begin{array}{l}
-\partial_t\phi+H(x,D\phi) \leq \alpha\qquad {\rm in }\; (0,T)\times \T^d\\
\phi(T^-,\cdot)=\phi_T \qquad {\rm in }\;  \T^d
\end{array}\right.
\ee
has a maximal subsolution. A subsolution of \eqref{subHJ} is a map $\phi\in BV$ such that $D\phi\in L^r$
and which satisfies inequality \eqref{subHJ} in the distributional sense in $(0,T)\times \T^d$. 

\begin{Lemma}\label{lem:maxsubsol} Assume that $\alpha\in L^p$ with $\alpha\geq 0$ a.e.. Then inequality
\eqref{subHJ} has a maximal subsolution $\bar \phi$. Namely $\phi\leq \bar \phi$ a.e. for any other subsolution $\phi$. 
Moreover, $\bar \phi$ is continuous in $[0,T]\times \T^d$ and locally Hölder continuous in $[0,T)\times \T^d$.
Finally, $\bar \phi$ satisfies  in the viscosity sense 
\be\label{supersol0}
-\partial_t \phi+ H(x,D\phi)\geq 0\qquad {\rm in }\; (0,T)\times \T^d\;.
\ee
\end{Lemma}

\begin{proof}  Let us fix $\alpha\in L^p$ with $\alpha\geq 0$ a.e. Let $\phi$ satisfy \eqref{subHJ} and $\psi$ be the viscosity solution to 
$$
\left\{\begin{array}{l}
-\partial_t \psi +H(x,D\psi)= 0\qquad {\rm in }\; (0,T)\times \T^d\\
\psi(T,x)=\phi_T(x)\qquad {\rm in }\;  \T^d
\end{array}\right.
$$
One easily checks that $\hat \phi= \phi\vee \psi$ is still a subsolution of \eqref{subHJ}. We claim that there is a constant $\Cl[C]{subHJ}>0$, independent of $\phi$, such that $\|D\hat \phi\|_r\leq \Cr{subHJ}$. Indeed, using the growth condition \eqref{HypGrowthH} on $H$ we have
$$
\begin{array}{rl}
\ds \int_0^T \int_{\T^d} \left( \frac{1}{r\Cr{C0H}} |D\hat \phi |^{r} -\Cr{C0H}\right)\;  \leq & \ds  \int_0^T \int_{\T^d}  H(x,D\hat \phi) \;
\leq \; \ds \int_0^T \int_{\T^d} \partial_t \hat \phi +  \alpha \\
\; \leq & \ds  \int_{\T^d} (\phi(T)- \psi(0)) + C
 \; \leq \; C\;.
\end{array}
$$
We now regularize $(\hat \phi,\alpha)$ by convolution:  let $\xi_\ep$ be as  in the proof of Proposition \ref{prop:valeursegales} and $\phi_\ep=\xi_\ep\star \hat \phi$. Then, by \eqref{phiepalphaep}, we have  
$$
-\partial_t \phi_\ep+H(x,D\phi_\ep) \leq \alpha_\ep \qquad {\rm in } \; (\ep,T-\ep)\times \T^d\;,
$$
where
$$
\alpha_\ep=\xi_\ep \star  \alpha+C  \ep^{1-(d+1)\theta/r}(1+ \Cr{subHJ}^\theta).
$$
Using Lemma \ref{lem:estiPhiT}, we also have that $\phi_\ep(T-\ep,x)\leq \phi_T(x)+ \Cr{estiPhiT}\Cr{subHJ}^\theta(2\ep)^\nu$. 

Let now $\tilde \phi_\ep$ be the viscosity solution of 
$$
\left\{\begin{array}{l}
\ds -\partial_t\phi+H(x,D\phi) = \alpha_\ep \qquad {\rm in }\; (0,T-\ep)\times \T^d\\
\ds \phi(T-\ep,\cdot)=\phi_T+ \Cr{estiPhiT}\Cr{subHJ}^\theta(2\ep)^\nu \qquad {\rm in }\;  \T^d
\end{array}\right.
$$
Note that $\tilde \phi_\ep$ is defined intrinsically and does not depend on the map $\phi$. In view of the estimate we proved on $\phi_\ep$, we have, by comparison, that $\tilde \phi_\ep\geq \phi_\ep$. 
Arguing as in the proof of Proposition \ref{Prop:existence}, one can check that the $(\tilde \phi_\ep)$ are uniformly bounded, that $(D\tilde \phi_\ep)$ is bounded in $L^r$, while $(H(x,D\tilde \phi_\ep) )$ is bounded in $L^1$.  Since $\partial_t \tilde \phi_\ep =H(x,D\tilde \phi_\ep)- \tilde \alpha_\ep$, the sequence $(\partial_t \tilde \phi_\ep)$  is bounded in $L^1$. This implies that $(\tilde \phi_\ep)$ is bounded in BV. Finally, as the $(\alpha_\ep)$ are bounded in $L^p$, Lemma \ref{CShold} states that the   $(\tilde \phi_\ep)$ are uniformly Hölder continuous in any compact subset of $[0,T)\times \T^d$. With all these estimates, one can show, as in  the proof of Proposition \ref{Prop:existence}, that a subsequence of the $(\tilde \phi_\ep)$ converges locally uniformly to a map $\tilde \phi$ which satisfies
\eqref{subHJ}. By construction, $\tilde \phi\geq \phi$ a.e.. This shows that $\tilde \phi$ is the maximal subsolution of \eqref{subHJ}. Moreover, as 
$$
 -\partial_t\tilde \phi_\ep+H(x,D\tilde \phi_\ep)= \alpha_\ep\geq 0\qquad {\rm in }\; (0,T-\ep)\times \T^d, 
 $$
the limit $\tilde \phi$ is also a viscosity supersolution of \eqref{supersol0}.  
\end{proof}

\section{Existence and uniqueness of a solution for the MFG system}\label{sec:exuniq}

In this section we show that the MFG system \eqref{MFG} has a unique weak solution. We first prove the existence, and then show that this solution is unique provided it satisfies an additional criterium. We complete the section by showing a stability property of the weak solution. 

\subsection{Definition of weak solutions} 

The variational method described above provides weak solutions for the MFG system. By a weak solution, we mean the following: 

\begin{Definition}\label{def:weaksolMFG} We say that a pair $(m, \phi)\in L^q((0,T)\times\T^d) \times BV((0,T)\times\T^d)$ is a weak solution to \eqref{MFG} if 
\begin{itemize}
\item[(i)] $\phi$ is continuous in $[0,T]\times \T^d$, with  
$$\ds D\phi\in L^r, \; \ds mD_pH(x,D\phi)\in L^1
\; {\rm and}\; 
\; \left(\partial_t\phi^{ac}- \lg D\phi, D_pH(x,D\phi)\rg\right)m \in L^1\;.$$ 

\item[(ii)] Equation \eqref{MFG}-(i) holds in the following sense:
\be\label{eq:ae}
\ds \quad -\partial_t \phi^{ac}(t,x) +H(x,D\phi(t,x))= f(x,m(t,x)) \quad \; \mbox{\rm a.e. in $\{m>0\}$}
\ee
 and  inequality
\be\label{eq:distrib}
\quad -\partial_t \phi +H(x,D\phi)\leq  f(x,m) \quad {\rm in}\; (0,T)\times \T^d 
\ee holds in the sense of distribution, with $\phi(T,\cdot)=\phi_T$ in the sense of trace,  

\item[(iii)] Equation \eqref{MFG}-(ii) holds:  
\be\label{eqcontdef}
\ds \quad \partial_t m-{\rm div}( mD_pH(x,D\phi))= 0\quad {\rm in }\; (0,T)\times \T^d, \qquad m(0)=m_0
\ee
in the sense of distribution,

\item[(iv)] The following equality holds: 
\be\label{defcondsup} 
\int_0^T\int_{\T^d} m\left(\partial_t\phi^{ac}- \lg D\phi, D_pH(x,D\phi)\rg \right)= \int_{\T^d} m(T)\phi_T-m_0\phi(0). 
\ee 
\end{itemize}
\end{Definition}

The definition, inspired by \cite{CCN},  requires some comments. First we note that the above (in)equalities have a meaning. Indeed, the growth condition \eqref{HypGrowthH} on $H$ together with assumption $D\phi\in L^r$ imply that the term $H(x,D\phi)$ is integrable. In the same way, as $m\in  L^q$ and $f$ has a growth given by \eqref{Hypf}, the term $f(\cdot,m(\cdot,\cdot))$ belongs to $L^p$, and, in particular, is integrable. Therefore requiring that \eqref{eq:distrib} holds in the sense of distribution has a sense. Analogously, the condition $mD_pH(x,D\phi)\in L^1$ ensures that \eqref{eqcontdef} makes sense, while the condition $\left(\partial_t\phi^{ac}- \lg D\phi, D_pH(x,D\phi)\rg\right)m \in L^1$ ensures the same holds for (iv). 

Next we note that condition (iii) gives the natural meaning to equation  \eqref{MFG}-(ii). The interpretation of \eqref{MFG}-(i) through condition (ii) is less obvious. Let us first point out that,  since $m$ is discontinuous, one cannot expect \eqref{MFG}-(i) to hold in a classical viscosity sense. Moreover, if first order Hamilton-Jacobi equations with a discontinuous right-hand side have been discussed in several papers (see e.g., \cite{CaSi}, \cite{ChH}, and the references therein), none of these references allows for a general form as \eqref{MFG}-(i).  Equality \eqref{eq:ae} is very close to requiring that \eqref{MFG}-(i) holds almost everywhere in $\{m>0\}$ (this would be the case if, for instance, $\phi$ was Lipschitz continuous---recall that $\partial_t \phi^{ac}$ denotes the absolutely continuous part of the measure $\partial_t \phi$). However, the meaning of the equation in the set $\{m=0\}$ is must less clear: inequality \eqref{eq:distrib} says that at least one inequality must hold.
 
We now discuss condition (iv).  When there is no regularity issue, i.e., when $\phi$ is smooth enough, condition (iv) is a simple consequence of (iii): just multiply \eqref{eqcontdef} by $\phi$ and integrate by parts to get 
$$
\int_0^T\int_{\T^d} m\left(\partial_t\phi- \lg D\phi, D_pH(x,D\phi)\rg \right)= \int_{\T^d} m(T)\phi_T-m_0\phi(0). 
$$
However, as $\partial_t\phi$ is a measure while $m$ is just integrable, the left-hand side of the above equality
has little meaning in general. So point (iv) explains that one can replace $\partial_t\phi$ by 
$\partial_t\phi^{ac}$ in the above expression. This roughly means that $\partial_t \phi^s=0$ in $\{m>0\}$. \\

Our main result is the following existence and uniqueness theorem:

\begin{Theorem}\label{theo:mainex} There exists a unique weak solution $(m,\phi)$ to the MFG system \eqref{MFG} which  satisfies  in the viscosity sense 
\be\label{CondSup}
-\partial_t \phi+ H(x,D\phi)\geq 0\qquad {\rm in }\; (0,T)\times \T^d\;.
\ee 
Moreover, the map $\phi$  is  locally Hölder continuous in $[0,T)\times \T^d$. 
\end{Theorem}

The existence part of the result relies on Theorem \ref{theo:main} below, which makes the link between weak solutions and the two optimization problems \eqref{Pb:mw2} and \eqref{PB:dual-relaxed}. Uniqueness cannot be expected in general: in fact, Theorem \ref{theo:unique} below explains that $m$ is always unique, but that $\phi$ is only determined on the set $\{m>0\}$. To have a full uniqueness result, on must add  condition \eqref{CondSup}: this condition is natural in the context, since one expects the right-hand side of \eqref{MFG}-(i) to be nonnegative. The proof of Theorem \ref{theo:mainex}---postponed to the end of subsection~\ref{subsec:uniq}---also shows that $\phi$ is the maximal solution of \eqref{subHJ} associated  with $\alpha=f(\cdot, m)$.

\subsection{Existence of a weak solution} 

The first step towards the proof of Theorem \ref{theo:mainex} consists in showing a one-to-one equivalence between solutions of the MFG system and the two optimizations problems \eqref{Pb:mw2} and \eqref{PB:dual-relaxed}.

\begin{Theorem}\label{theo:main}
If $(m,w)\in \mathcal K_1$ is a minimizer of \eqref{Pb:mw2} and $(\phi, \alpha)\in \mathcal K$ is a minimizer of \eqref{PB:dual-relaxed} such that $\phi$ is continuous, then $(m,\phi)$ is a solution of the mean field game system \eqref{MFG} and $w= -mD_pH(\cdot,D\phi)$ while $\alpha= f(\cdot,m)$ a.e.. 

Conversely, any weak solution of \eqref{MFG}  is such that the pair $(m,-mD_pH(\cdot,D\phi))$ is the minimizer of \eqref{Pb:mw2} while $(\phi, f(\cdot,m))$ is a minimizer of \eqref{PB:dual-relaxed}. 
\end{Theorem}

 The proof of Theorem \ref{theo:main} requires a preliminary Lemma:

\begin{Lemma}\label{Lem: ineqq} Let $(m,w) \in L^q((0,T)\times \T^d) \times L^1((0,T)\times \T^d,\R^d)$ satisfy the continuity equation
$$\quad \partial_t m+{\rm div} (w)=0\quad {\rm in}\; (0,T)\times \T^d,
\qquad m(0)=m_0
$$ and be such that $\ds mH^*(\cdot, -\frac{w}{m}) \in L^1$ and let $(\phi,\alpha)\in {\mathcal K}$ with $\alpha\geq0$ a.e.. Then 
\be\label{ineqq}
\int_0^T \int_{\T^d} \left(\alpha+  H^*(x, -\frac{w}{m})\right)m + \int_{\T^d} \phi_Tm(T)-\phi(0)m_0 \; \geq \; 0\;.
\ee
Moreover, if $\phi$ is continuous in $[0,T]\times \T^d$ and if equality holds in \eqref{ineqq}, then 
$$w(t,x)=-m(t,x)D_pH(x,D\phi(t,x))\; {\rm a.e.}
$$
\end{Lemma}

\begin{proof} As $\ds mH^*(\cdot, -\frac{w}{m}) \in L^1$, the growth condition \eqref{HypHstar} on $H^*$  implies that $v:= w/m$ belongs to $L^{r'}((0,T)\times \T^d, m)$. So  
the continuity equation can be rewritten as 
$$\quad \partial_t m+{\rm div} (mv)=0\quad {\rm in}\; (0,T)\times \T^d,
\qquad m(0)=m_0
$$
In particular, standard results for this equation (see, e.g., \cite{AGS}) imply that  $m\geq 0$ a.e. and that $t\to m(t)$ is continuous from $[0,T]$ to $P(\T^d)$. 

It is clearly enough to show that \eqref{ineqq} holds for the maximal subsolution of \eqref{subHJ} defined in Lemma \ref{lem:maxsubsol}.  Recall that this maximal subsolution is continuous in $[0,T]\times \T^d$, which is all we shall need here. 
 Fix $\delta>0$ small and let $\xi$ be the convolution kernel as defined in the proof of Proposition \ref{prop:valeursegales}. We set $\xi_\ep(t,x)= \ep^{-d-1} \xi((t,x)/\ep)$ and $\phi_\ep=\xi_\ep \star \phi$. By \eqref{phiepalphaep}, we have 
$$
-\partial_t \phi_\ep+H(x,D\phi_\ep) \leq \alpha_\ep  \qquad {\rm in } \; (\ep,T-\ep)\times \T^d\;,
$$
where $\alpha_\ep=\xi_\ep \star \alpha+ C  \ep^{1-(d+1)\theta/r}(1+ \|D\phi\|_r^\theta)$. Then, for $0<\ep<\delta$,  
$$
\begin{array}{rl}
\ds \int_\delta^{T-\delta}\int_{\T^d} \left(\alpha_\ep+  H^*(x, -\frac{w}{m})\right)m \; \geq & 
\ds \int_\delta^{T-\delta}\int_{\T^d} \left(-\partial_t \phi_\ep+H(x,D\phi_\ep)+  H^*(x, -\frac{w}{m})\right)m \\
\geq & \ds \int_\delta^{T-\delta}\int_{\T^d} \left(-m \partial_t \phi_\ep-\lg D\phi_\ep, w\rg \right) \\
\geq & \ds - \int_{\T^d} \left(\phi_\ep(T-\delta)m(T-\delta)-\phi_\ep(\delta)m(\delta)\right) 
\end{array}
$$
since the pair $(m,w)$ satisfies the continuity equation. 
Letting $\ep\to 0$ we get, by continuity of $\phi$, 
$$
\int_\delta^{T-\delta}\int_{\T^d} \left(\alpha+  H^*(x, -\frac{w}{m})\right)m \; \geq\; 
- \int_{\T^d} \left(\phi(T-\delta)m(T-\delta)-\phi(\delta)m(\delta)\right).
$$
Using again the continuity of $\phi$ and the continuity of the map $t\to m(t)$ for the weak-* convergence of measures, we obtain \eqref{ineqq}. 

Let us now assume $\phi$ is continuous in $[0,T]\times \T^d$ and that equality holds in \eqref{ineqq}. We first claim that 
\be\label{H+H*=}
\left(H(x,D\phi)+  H^*(x, -\frac{w}{m})\right)m = -\lg D\phi, w\rg \qquad \mbox{\rm a.e. in }(0,T)\times \T^d.
\ee
Indeed, assume that \eqref{H+H*=} does not hold. Then there are $\theta>0$, $\delta>0$ with 
$$
\int_{\delta}^{T-\delta} \inte \left[\left(H(x,D\phi)+  H^*(x, -\frac{w}{m})\right)m + \lg D\phi, w\rg \right] \geq \theta.
$$
Since $D\phi_\ep\to D\phi$ a.e. as $\ep\to 0$, we get by Fatou
$$
\int_{\delta}^{T-\delta} \inte \left[ \left(H(x,D\phi_\ep)+  H^*(x, -\frac{w}{m})\right)m + \lg D\phi_\ep, w\rg \right] \geq \theta/2
$$
for $\ep>0$ sufficiently small. Applying the construction of the first part of the proof (where we only used the continuity of $\phi$), we obtain therefore
$$
\int_0^T \int_{\T^d} \left(\alpha+  H^*(x, -\frac{w}{m})\right)m + \int_{\T^d} \phi_Tm(T)-\phi(0)m_0 \; \geq \; \theta/2\;,
$$
which contradicts our assumption. So \eqref{H+H*=} holds, which implies that $$\ds w(t,x)= -m(t,x)D_pH(x,D\phi(t,x))\qquad\mbox{\rm a.e. in $\{m>0\}$.}
$$
 By the coercivity assumption \eqref{HypHstar} on $H^*$, $\ds mH^*(\cdot, -\frac{w}{m}) \in L^1$ implies that $w=0$ a.e. in $\{m=0\}$. Therefore $\ds w(t,x)= -m(t,x)D_pH(x,D\phi(t,x))$ also holds a.e. in $\{m=0\}$.
\end{proof}

\begin{proof}[Proof of Theorem \ref{theo:main}] Let $(m,w)\in L^q((0,T)\times \T^d) \times L^1((0,T)\times \T^d,\R^d)$ be a solution of \eqref{Pb:mw2} and $(\phi,\alpha)\in {\mathcal K}$ be a solution of the relaxed problem \eqref{PB:dual-relaxed} given by Proposition \ref{Prop:existence}. Recall that $\ds mH^*(\cdot, -\frac{w}{m}) \in L^1$, that $\phi$ is continuous in $[0,T]\times \T^d$ and $\ds \alpha=\left( -\partial_t\phi^{ac} +H(x,D\phi)\right)\vee 0$ a.e.. From Lemma \ref{Lem:dualite} and Proposition \ref{prop:valeursegales}, we have 
$$
0=  \int_0^T\int_{\T^d} m H^*(x, -\frac{w}{m})+ F(x,m)+ F^*(x,\alpha) + \int_{\T^d} \phi_Tm(T)-\phi(0)m_0\;.
$$
Since $m\in L^q$ while $\alpha\in L^p$, we also have 
$$
\begin{array}{l}
\ds\int_0^T\int_{\T^d} m H^*(x, -\frac{w}{m})+ F(x,m)+ F^*(x,\alpha) + \int_{\T^d} \phi_Tm(T)-\phi(0)m_0\\
\qquad \qquad 
\ds\geq \; \int_0^T\int_{\T^d} \left( H^*(x, -\frac{w}{m})+ \alpha\right)m+ \int_{\T^d} \phi_Tm(T)-\phi(0)m_0\; \geq \; 0,
\end{array}
$$
where the last inequality comes from Lemma \ref{Lem: ineqq}.  Since equality holds in the above string of inequalities, one must have 
$$
F(x,m)+ F^*(x,\alpha) =m\alpha \qquad {\rm a.e.,}
$$
i.e., $\alpha(t,x)=f(x,m(t,x))$ a.e. and, from the second statement of Lemma \ref{Lem: ineqq}, $w(t,x)=-m(t,x)D_pH(x,D\phi(t,x))$ a.e. In particular $mD_pH(\cdot,D\phi)\in L^1$. Note that $\{\alpha>0\}=\{m>0\}$ and therefore  (ii) and (iii) in Definition \ref{def:weaksolMFG} hold. Using again that $\{\alpha>0\}=\{m>0\}$ and that $\ds \alpha=\left( -\partial_t\phi^{ac} +H(x,D\phi)\right)\vee 0$ a.e., we have
$$
\begin{array}{rl}
\ds  \left( H^*(x, -\frac{w}{m})+ \alpha\right)m\; = & \ds \left( H^*(x, D_pH(x,D\phi))-\partial_t\phi^{ac}+H(x,D\phi) \right)m\\
 = & \ds \left( -\partial_t\phi^{ac} + \lg D\phi, D_pH(x,D\phi)\rg  \right)m.
\end{array}
$$
As the left-hand side belongs to $L^1$, so does the right-hand side: this completes the proof of (i). Then equality 
$$
\int_0^T\int_{\T^d} \left( H^*(x, -\frac{w}{m})+ \alpha\right)m+ \int_{\T^d} \phi_Tm(T)-\phi(0)m_0\; = \; 0
$$
can be rewritten as 
$$
\int_0^T\int_{\T^d} \left(  \partial_t\phi^{ac} - \lg D\phi, D_pH(x,D\phi)\rg \right)m+ \int_{\T^d} \phi_Tm(T)-\phi(0)m_0\; = \; 0\;.
$$
So (iv) holds as well. In conclusion, the pair $(m,\phi)$ is a weak solution of \eqref{MFG}. 


Let us now assume that $(m,\phi)$ is a weak solution of  \eqref{MFG}. Let us set $w=-mD_pH(x,D\phi)$ and $\alpha=f(x,m)$. Then $(m,w)$ belongs to $\mathcal K_1$ and $(\phi,\alpha)\in \mathcal K$ by (i), (ii) and (iii) in Definition \ref{def:weaksolMFG}. 
We first prove that $(m,w)$ is optimal for \eqref{Pb:mw2}. Recall that $m\in L^q$  by  definition of a weak solution. In view of the growth condition \eqref{Hypf} we have therefore $f(\cdot,m(\cdot,\cdot))\in L^p$.  Let $(m',w')\in \mathcal K_1$ be another admissible pair. Without loss of generality we can assume that $m'H^*(x,-\frac{w'}{m'})\in L^1$ and $m'\in L^q$, because otherwise ${\mathcal B}(m',w')=+\infty$. Then, by convexity of $F$ with respect to the second variable, we have: 
$$
\begin{array}{rl}
\ds {\mathcal B}(m',w')\; =& \ds \int_0^T\int_{\T^d} m' H^*(x, -\frac{w'}{m'})+ F(x,m') + \int_{\T^d} \phi_T m'(T) \\
\geq &\ds \ds \int_0^T\int_{\T^d} m' H^*(x, -\frac{w'}{m'})+ F(x,m) +  f(x,m)(m'-m) + \int_{\T^d} \phi_T m'(T) 
 \end{array}
 $$
Next we use the definition of $\alpha$ and the fact that $(\phi,\alpha)\in \mathcal K$ to get: 
$$
\begin{array}{rl}
\ds {\mathcal B}(m',w') \;  \geq &  \ds \int_0^T\int_{\T^d} m' (\alpha+ H^*(x, -\frac{w'}{m'})) + F(x,m)-f(x,m) m + \int_{\T^d} \phi_T m'(T) \\
 \geq & \ds \int_0^T\int_{\T^d}  F(x,m)-f(x,m) m + \int_{\T^d} \phi(0) m_0 
 \end{array}
 $$
 where the last inequality comes from the first statement of Lemma \ref{Lem: ineqq}. Using conditions (ii) and (iv) in Definition \ref{def:weaksolMFG}, we have 
 $$
\begin{array}{rl}
\ds  - \int_0^T\int_{\T^d} f(x,m) m + \int_{\T^d} \phi(0) m_0  \; = & \ds \int_0^T \int_{\T^d} m\left(\partial_t\phi^{ac} -H(x,D\phi) \right) + \int_{\T^d} \phi(0) m_0\\
= & \ds   \int_0^T \int_{\T^d} m\left(\lg D\phi, D_pH(x,D\phi)\rg -H(x,D\phi) \right) + \int_{\T^d} \phi_T m(T)\\
= & \ds   \int_0^T mH^*(x,D_pH(x,D\phi)) + \int_{\T^d} \phi_T m(T)\\
= & \ds   \int_0^T m H^*(x,-\frac{w}{m})  + \int_{\T^d} \phi_T m(T)\\
 \end{array}
  $$
Therefore
$$
{\mathcal B}(m',w') \geq \int_0^T\int_{\T^d} mH^*(x, -\frac{w}{m})+ F(x,m) + \int_{\T^d} \phi_T m(T)= {\mathcal B}(m,w)\;,
 $$
 which proves the optimality of $(m,w)$.  \\
 
The arguments for proving the optimality of $(\phi,\alpha)$ are similar: we already know that  $(\phi,\alpha)$ belongs to ${\mathcal K}$. Let $(\phi',\alpha')\in {\mathcal K}$ be another admissible test function. From Proposition \ref{prop:valeursegales} we can assume without loss of generality that $\phi'$ is of class ${\mathcal C}^1$ and $\alpha'= -\partial_t\phi'+H(x,D\phi')$. Then, since $m\in \partial_\alpha   F^*(x,\alpha)$ because $\alpha = f(x,m)$, we have
$$
\ds \mathcal A(\phi',\alpha')=  \int_0^T\int_{\T^d}  F^*(x,\alpha') - \int_{\T^d} \phi'(0)m_0 \; \geq \;
\int_0^T\int_{\T^d}  F^*(x,\alpha) + m(\alpha'-\alpha) - \int_{\T^d} \phi'(0)m_0.
$$
From the first statement of Lemma \ref{Lem: ineqq}, we have  
$$
\int_0^T\int_{\T^d}  m\alpha' - \int_{\T^d} \phi'(0)m_0 \geq 
- \int_0^T \int_{\T^d}  m H^*(x, -\frac{w}{m}) - \int_{\T^d} \phi_Tm(T) .
$$
So 
$$
\ds \mathcal A(\phi',\alpha') \; \geq \;
\int_0^T\int_{\T^d}  F^*(x,\alpha) -m H^*(x, -\frac{w}{m})-m \alpha- \int_{\T^d} \phi_Tm(T)  .
$$
Using the definition of $w$ and $\alpha$ and condition (iv) in Definition \ref{def:weaksolMFG}, we have
$$
\begin{array}{rl}
\ds \int_0^T\int_{\T^d}  m H^*(x, -\frac{w}{m})+ m \alpha \; = & \ds \int_0^T\int_{\T^d}  m \left(H^*(x, -\frac{w}{m})-\partial_t\phi^{ac} +H(x,D\phi)\right)\\
 = &  \ds \int_0^T\int_{\T^d}  m \left( \lg D\phi, D_pH(x,D\phi)\rg -\partial_t\phi^{ac} \right)\\
= & \ds - \int_{\T^d} m(T)\phi_T-m_0\phi(0) 
 \end{array}
$$
Therefore 
$$
{\mathcal A}(\phi',\alpha') \; \geq \;\int_0^T\int_{\T^d}  F^*(x,\alpha) - \int_{\T^d} \phi(0)m_0= {\mathcal A}(\phi,\alpha) ,
$$
which proves the optimality of $(\phi,\alpha)$. 
\end{proof}

\subsection{Uniqueness of the weak solution}\label{subsec:uniq}


\begin{Theorem}\label{theo:unique} Let $(m,\phi)$ and $(m',\phi')$ be two weak solutions of \eqref{MFG}. Then $m=m'$ in $[0,T]\times \T^d$ while $\phi=\phi'$ in $\{m>0\}$. 
%
\end{Theorem}


The proof of Theorem \ref{theo:unique} requires several steps and relies on a representation of solutions in terms of measures over family of curves. 

Let $(m,\phi)$ be a solution to \eqref{MFG}. In view of Theorem \ref{theo:main},  the pair $(m,-mD_pH(\cdot,D\phi))$ is the minimizer of \eqref{Pb:mw2} while $(\phi, f(\cdot,m))$ is a solution of \eqref{PB:dual-relaxed}. In particular,  $m$ is unique because of the uniqueness of the solution of \eqref{Pb:mw2}.

Let now $\alpha= f(\cdot, m)$ and $\bar \phi$ be the maximal subsolution of \eqref{subHJ} associated with $\alpha$. Note that $\bar \phi$ is defined independently of $\phi$. So, in order to show Theorem \ref{theo:unique}, we just need to prove that $\phi$ coincides with $\bar \phi$ in the set $\{m>0\}$. In view of Lemma \ref{lem:maxsubsol}, we have $\bar \phi\geq \phi$ a.e.. Therefore, the pair $(\bar \phi, \alpha)$ is also a minimizer to \eqref{PB:dual-relaxed}. Note that this implies that 
\be\label{phi0=psi0mae}
\phi(0,\cdot)=\bar \phi(0, \cdot) \qquad \mbox{\rm a.e. in }\{m_0>0\}\;.
\ee
Note also that, according to Theorem \ref{theo:main},  the pair $(\bar \phi,\alpha)$ is also a solution of \eqref{MFG}. \\


%
%

Let $\Gamma$ be the set of continuous curves $\gamma:[0,T]\to \T^d$ endowed with the topology of uniform convergence. We consider the set $M(\Gamma)$  of Borel probability measures on $\Gamma$. For any $t\in [0,T]$, we denote by $e_t:\Gamma\to \T^d$ the evaluation map: $e_t(\gamma)=\gamma(t)$. We are particularly interested in the subset $\widetilde M(\Gamma)$ of measures $\eta'\in M(\Gamma)$  such that 
$$
\int_{\Gamma} \int_0^T \left|\dot \gamma(s)\right|^{r'} \ dsd \eta'(\gamma)<+\infty
$$
and such that $m'(t):= e_t\sharp \eta$ is absolutely continuous for any $t\in [0,T]$ (the density being also denoted by $m'(t,\cdot)$), with $m'\in  L^q((0,T)\times\T^d)$. 

Throughout the section, it will be convenient to  denote by $L$ the convex conjugate of the map $p\to H(x,-p)$, i.e., $L(x,\xi)=H^*(x,-\xi)$. Recall that $(m,\phi)$ is a solution to \eqref{MFG} and that $\alpha= f(\cdot, m)$. 

\begin{Lemma}\label{lem:t1t2} Let $\eta'\in \widetilde M(\Gamma)$ and set $m'(t)=e_t\sharp \eta'$. We have, for any $0\leq t_1< t_2\leq T$, 
\be\label{ineq:lem:t1t2}
\begin{array}{rl}
\ds \int_{\T^d} \phi(t_1,x) m'(t_1,x)\ dx \; \leq & \ds \int_{\T^d} \phi(t_2,x) m'(t_2,x)\ dx 
+ \int_\Gamma\int_{t_1}^{t_2} L(\gamma(s),\dot\gamma(s)) \ dsd\eta'(\gamma) \\
& \ds \qquad \qquad +\int_{\T^d}\int_{t_1}^{t_2} \alpha(s, x) m'(s,x)\ dsdx .
\end{array}
\ee
\end{Lemma}

\begin{proof} As $\phi$ is continuous, we can assume that $0<t_1<t_2<T$. We regularize $\phi$ and $\alpha$ into $\phi^\ep$ and $\alpha^\ep$ as in the proof of Proposition \ref{prop:valeursegales}. Since, for $\ep$ small enough,  inequality $-\partial_t \phi^\ep+H(x,D\phi^\ep))\leq \alpha^\ep$ holds in $(t_1, t_2)\times \T^d$, we have for any $\gamma \in W^{1, r'}([0,T])$:
$$
\begin{array}{l}
\ds \frac{d}{dt}\left[ \phi^\ep(s,\gamma(s))-\int_s^T L(\gamma(\tau), \dot \gamma(\tau))d\tau\right] \\ 
\qquad \qquad = \; \ds \partial_t \phi^\ep(s,\gamma(s))+ \lg D\phi^\ep(s,\gamma(s)), 
\dot \gamma(s)\rg +L(\gamma(s), \dot \gamma(s))\\ 
\qquad \qquad  \geq \;  \ds  \partial_t \phi^\ep(s,\gamma(s))- H(\gamma(s), D\phi^\ep(s,\gamma(s))) \;
\geq \;  \ds -\alpha^\ep(s,\gamma(s))
\end{array}
$$
We integrate first between $t_1$ and $t_2$ and then over $\eta'$ to get
$$
\begin{array}{rl}
\ds \int_{\T^d}\phi^\ep(t_1,x)m'(t_1,x)dx \; \leq &
\ds \int_{\T^d}\phi^\ep(t_2,x)m'(t_2,x)dx + 
\int_{\Gamma}\int_{t_1}^{t_2} L(\gamma(s), \dot \gamma(s))\ ds d\eta(\gamma)\\
& \ds \qquad \qquad +\int_{\T^d}\int_{t_1}^{t_2} \alpha^\ep(s,x)m'(s,x)\  dsdx 
\end{array}
$$
Letting $\ep\to 0$ we obtain the desired inequality, since  $\phi^\ep$ converges to $\phi$ uniformly and $\alpha^\ep$ converges to $\alpha$ in $L^{p}$, with $m'\in  L^q((0,T)\times\T^d)$
\end{proof}

We now build a specific measure $\eta$ for which equality holds in \eqref{ineq:lem:t1t2}.  Let us set, as usual, $w= -mD_pH(\cdot, D\bar \phi)$. Recall that $(m,w)$ is a solution of \eqref{Pb:mw2}. Let $\xi$ be a standard convolution kernel in $\R^d$ such that $\xi>0$ in $\R^d$ and let $m^\ep=\xi_\ep\star m$, $w^\ep= \xi^\ep\star w$. We note  that $m^\ep>0$. For $x\in \T^d$, let $X^\ep_x$ be the solution to the Cauchy problem
$$
\left\{ \begin{array}{l}
\ds \dot x(t)=  \frac{w^\ep(t,x(t))}{m^\ep(t,x(t))}  \qquad \mbox{\rm a.e. in}\;  [0,T]\\
x(0)=x
\end{array}\right.
$$
We define $\eta^\ep\in M(\Gamma)$ by 
$$
\int_\Gamma \Theta(\gamma) d\eta^\ep(\gamma)= \int_{\T^d} \Theta(X^\ep_x)m_0^\ep(x)dx
$$
for any bounded, continuous map $\Theta:\Gamma\to \R$. One easily checks that $m^\ep(t)= e_t\star \eta^\ep= X^\ep_\cdot (t)\sharp m_0^\ep$ (i.e., $m^\ep(t)$ is the push forward of the measure $m_0$ by the map $x\to X^\ep_x (t)$. In particular, $\eta^\ep\in \widetilde M(\Gamma)$.

\begin{Lemma} The family $(\eta^\ep)$ is tight.
\end{Lemma}

\begin{proof} Let us $\Phi:\Gamma\to \R\cup\{+\infty\}$ be defined by  
\be\label{defPhi}\ds 
 \Phi(\gamma)= \left\{\begin{array}{ll}
 \ds  \int_0^T L(\gamma(t),\dot \gamma(t))dt & {\rm if }\; \gamma\in W^{1,r'}([0,T]\\
 \ds +\infty & {\rm otherwise}
 \end{array}\right. 
 \ee
 Then $\Phi$ is  lower semicontinuous, convex and coercive thanks to assumption \eqref{HypHstar}. 
 We have, by definition of $\eta^\ep$, 
 \be\label{Phi1}
\begin{array}{rl}
\ds\int_{\Gamma}\Phi(\gamma) d\eta^\ep(\gamma)\; = &  
\ds \int_{\T^d}\int_0^T L\left(X^\ep_x(t),\frac{w^\ep(t,X^\ep_x(t))}{m^\ep(t,X^\ep_x(t))}\right)m_0^\ep(x)\ dtdx\\
= & \ds \int_0^T\int_{\T^d} H^*\left(x,-\frac{w^\ep(t,x)}{m^\ep(t,x)}\right) m^\ep(t,x)\ dx dt
\end{array}
\ee
Note that, by convexity of the map $(m,w)\to H^*(x,  -\frac{w}{m})m$, 
\be\label{Phi2}
\ds \limsup_{\ep\to 0} \int_0^T\int_{\T^d} H^*\left(x,-\frac{w^\ep(t,x)}{m^\ep(t,x)}\right) m^\ep(t,x)dx dt
\leq  \int_0^T\int_{\T^d} H^*\left(x,-\frac{w(t,x)}{m(t,x)}\right) m(t,x)dt.
\ee
Since the right-hand side of the above inequality is finite, $\ds \int_{\Gamma}\Phi(\gamma)d\eta^\ep(\gamma)$ is uniformly bounded. As $\Phi$ has compact level-set in $\Gamma$, this implies that $\eta^\ep$ is tight. 
\end{proof}

Let $\eta$ be a limit of a subsequence of the $(\eta^\ep)$. Recall that $\alpha= f(\cdot, m)$ and that $\bar \phi$ be the maximal subsolution of \eqref{subHJ} associated with $\alpha$. 

\begin{Lemma}\label{lem:equalpsi} We have $m(t)= e_t\sharp \eta$ for any $t\in [0,T]$ and
\be\label{equalpsi}
\begin{array}{rl}
\ds \int_{\T^d} \bar \phi(0,x)m(0,x) dx \; = & \ds \int_{\T^d} \bar \phi_T(x)m(T,x)+ \int_\Gamma\int_0^T L(\gamma(s),\dot\gamma(s)) dsd\eta(\gamma) \\
& \ds \qquad \qquad +\int_{\T^d}\int_0^T \alpha(s, x) m(s,x)dsdx 
\end{array}
\ee
\end{Lemma}

\begin{Remark}{\rm Since $e_0\sharp \eta=m_0$, by desintegration there exists a Borel measurable family of probabilities $(\eta_x)_{x\in \T^d}$ on $\Gamma$ such that $\ds \eta(d\gamma)= \int_{\T^d}\eta_x(d\gamma) m_0(dx)$ and, for $m_0-$a.e. $x\in \T^d$, $\eta_x-$almost any trajectory $\gamma$ starts at $x$. Heuristically, combination of Lemma \ref{lem:t1t2} and Lemma \ref{lem:equalpsi} says that the measure $\eta_x$ is supported by optimal trajectories for the optimal control problem
$$
\inf_{\gamma(0)=x} \int_0^T \left(L(\gamma(s),\dot \gamma(s))+ \alpha(s,\gamma(s))\right)ds + g(\gamma(T))
$$
and that $\phi$ is the value function associated with this problem. 
Of course this statement is meaningless because the map $\alpha$ is not regular enough to define the above quantity. 
}\end{Remark}

\begin{proof}[Proof of Lemma \ref{lem:equalpsi}.] We first check equality $m(t)= e_t\sharp \eta$ for any $t\in [0,T]$. Let $h\in C^0([0,T]\times \T^d)$. Then 
$$
\begin{array}{rl}
\ds \int_0^T\int_{\T^d} h(t,x)d(e_t\sharp\eta)(x)dt\; =  & \ds \int_\Gamma\int_0^T h(t,\gamma(t))dt d\eta(\gamma)
= \lim_{\ep\to 0} \int_\Gamma\int_0^T h(t,\gamma(t))dt d\eta^\ep(\gamma) \\
= & \ds \lim_{\ep\to 0} \int_{\T^d}\int_0^T h(t,x)m^\ep(t,x) dtdx = \int_{\T^d}\int_0^T h(t,x)m(t,x) dtdx 
\end{array}
$$
This proves the equality  $m(t)= e_t\sharp \eta$ for a.e. $t\in [0,T]$, and therefore for any $t$ by continuity of $m$ and $e_t\sharp \eta$ in $P(\T^d)$. 

Next we show \eqref{equalpsi}. Recall that $(m,\bar \phi)$ is a weak solution of \eqref{MFG}. 
In view of \eqref{eq:ae}, equality \eqref{defcondsup} can be rewritten as 
$$
\int_0^T\int_{\T^d} m\left(H(x,D\bar \phi)-\alpha-\lg D\bar \phi, D_pH(x,D\bar \phi)\rg \right)= \int_{\T^d} m(T)\phi_T-m_0\bar \phi(0),
$$
where, by definition of the convex conjugate, 
$$
H(x,D\bar \phi)-\lg D\bar \phi, D_pH(x,D\bar \phi)\rg  = -H^*(x,D_pH(x,D\bar \phi)).
$$
So, by definition of $w$, we have 
$$
\int_0^T\int_{\T^d} m\left(\alpha+H^*(x,-\frac{w}{m}) \right)+ \int_{\T^d} m(T)\phi_T-m_0\bar \phi(0)=0.
$$
On another hand,  \eqref{Phi1} and \eqref{Phi2} imply that
$$
\limsup_{\ep\to0} \int_\Gamma\int_0^T L(\gamma(t),\dot\gamma(t)) dt d\eta^\ep(\gamma)
\leq  \int_0^T\int_{\T^d} H^*\left(x,-\frac{w}{m}\right) m\ dxdt, 
$$
where, by lower semi-continuity of $\Phi$ defined in \eqref{defPhi}, 
$$
 \int_\Gamma\int_0^T L(\gamma(t),\dot\gamma(t)) dt d\eta(\gamma)\leq 
\limsup_{\ep\to0} \int_\Gamma\int_0^T L(\gamma(t),\dot\gamma(t)) dt d\eta^\ep(\gamma).
$$
Putting together the three above inequalities, we get 
$$
 \int_\Gamma\int_0^T L(\gamma(t),\dot\gamma(t)) dt d\eta(\gamma)
+ \int_0^T\int_{\T^d} m\alpha+ \int_{\T^d} m(T)\phi_T-m_0\bar \phi(0)\leq0.
$$
Using finally Lemma \ref{lem:t1t2} yields  the desired result. 
%
\end{proof}

We are now ready to complete the proof of  Theorem \ref{theo:unique}.

\begin{proof}[Proof of  Theorem \ref{theo:unique}] We have already established the uniqueness of $m$. It remains to show that, for any $t\in [0,T]$, we have $\phi(t, \cdot)= \bar \phi(t,\cdot)$ a.e. on $\{m(t,\cdot)>0\}$, where $\bar \phi$ is the maximal solution of \eqref{subHJ} associated with $\alpha$. 
 We know that the result already holds for $t=T$ (because $\phi(T,\cdot)=\bar \phi(T,\cdot)=\phi_T$) and $t=0$ thanks to 
\eqref{phi0=psi0mae}. Fix $t\in (0,T)$. 
We apply Lemma \ref{lem:t1t2} to $\eta$ twice, first with $t_1=0$ and $t_2=t$ and then with $t_1=t$ and $t_2=T$: we have 
$$
\begin{array}{rl}
\ds \int_{\T^d} \phi(0,x) m(0,x)\ dx \; \leq & \ds \int_{\T^d} \phi(t,x) m(t,x)\ dx 
+ \int_\Gamma\int_{0}^{t} L(\gamma(s),\dot\gamma(s)) \ dsd\eta(\gamma) \\
& \ds \qquad \qquad +\int_{\T^d}\int_{0}^{t} \alpha(s, x) m(s,x)\ dsdx 
\end{array}
$$
and 
\be\label{ineququ2}
\begin{array}{rl}
\ds \int_{\T^d} \phi(t,x) m(t,x)\ dx \; \leq & \ds \int_{\T^d} \phi_T(x) m(T,x)\ dx 
+ \int_\Gamma\int_{t}^{T} L(\gamma(s),\dot\gamma(s)) \ dsd\eta(\gamma) \\
& \ds \qquad \qquad +\int_{\T^d}\int_{t}^{T} \alpha(s, x) m(s,x)\ dsdx 
\end{array}
\ee
We add both inequalities to get
$$
\begin{array}{rl}
\ds \int_{\T^d} \phi(0,x) m(0,x)\ dx \; \leq & \ds \int_{\T^d} \phi_T(x) m(T,x)\ dx 
+ \int_\Gamma\int_{0}^{T} L(\gamma(s),\dot\gamma(s)) \ dsd\eta(\gamma) \\
& \ds \qquad \qquad +\int_{\T^d}\int_{0}^{T} \alpha(s, x) m(s,x)\ dsdx 
\end{array}
$$
Since $\phi(0,\cdot)=\bar \phi(0, \cdot)$ a.e. on $\{m_0>0\}$, 
Lemma \ref{lem:equalpsi} states that the above inequality is in fact an equality. This implies in particular that there is an equality in \eqref{ineququ2}. Since the right-hand side of \eqref{ineququ2} does not depend of the specific choice of the minimizer, we get $\ds \int_{\T^d} \bar \phi(t,x) m(t,x)\ dx= \int_{\T^d} \phi(t,x) m(t,x)\ dx$. As  $\phi\leq \bar \phi$, this implies that $\bar \phi(t,\cdot)=\phi(t,\cdot)$ in $\{m(t,\cdot)>0\}$. 

\end{proof}

\begin{proof}[Proof of Theorem \ref{theo:mainex}] Proposition \ref{Prop:existence} states that  there is a solution $(\phi, \alpha)$ of the relaxed problem \eqref{PB:dual-relaxed}  such that $\phi$ is  locally Hölder continuous in $[0,T)\times \T^d$ and satisfies \eqref{CondSup}  in the viscosity sense. 
So Theorem \ref{theo:main} readily implies the existence part of Theorem \ref{theo:mainex}.

We now assume that $(m,\phi)$ is a solution of \eqref{MFG} for which $\phi$ satisfies \eqref{CondSup}. Let ${\mathcal O}=\{\phi<\bar \phi\}$. By Theorem \ref{theo:unique}, $m=0$ a.e. in the open set ${\mathcal O}$. So $\bar \phi$ solves  $-\partial_t\bar \phi+H(\cdot,D\bar \phi)\leq 0$ in ${\mathcal O}$ in the sense of distribution. The Hamiltonian being continuous and convex in the second variable, this inequality also holds in the viscosity sense.  On another hand, $\phi$ solves in the viscosity sense $-\partial_t\phi+H(\cdot,D\phi)\geq 0$ in $[0,T]\times \T^d$ and therefore in ${\mathcal O}$. But $\phi=\bar \phi$ in $\partial {\mathcal O}$, so that, by comparison,  $\phi\geq \bar \phi$ in ${\mathcal O}$. Since inequality $\phi\leq \bar \phi$ always holds by construction, we get $\phi=\bar \phi$ and uniqueness holds.  
\end{proof}


\subsection{Stability}

We complete this section by a  stability property of the weak solution of \eqref{MFG}. Assume that $(\phi^n,m^n)$ is the unique weak solution of \eqref{MFG} associated with an Hamiltonian $H^n$, a coupling $f^n$ and with the initial and terminal conditions $m_0^n$ and $\phi_T^n$, such that $\phi^n$ satisfies the additional condition \eqref{CondSup}. We suppose that the $(H^n)$, $(f^n)$, $(m_0^n)$ and $(\phi_T^n)$ satisfy the conditions (H1)$\dots$(H4) with rate growth and constants independent of $n$ and converge locally uniformly to $H$, $f$, $m_0$ and $\phi_T$ respectively. 

\begin{Proposition}\label{Prop:stabilo} The  $(\phi^n,m^n)$  converge, respectively locally uniformly and in $L^q$, to the unique solution $(\phi,m)$ of \eqref{MFG} associated with $H$, $f$, $m_0$ and $\phi_T$ for which \eqref{CondSup} holds.
\end{Proposition}

The result is a simple consequence of Theorem \ref{theo:main} and of the $\Gamma-$convergence of the corresponding variational problems.  

\begin{proof} Let us set $w^n=-m^nD_pH_n(\cdot,D\phi^n))$ and $\alpha^n= f(\cdot,m^n)$. According to the second part of Theorem \ref{theo:main}, the pair $(m^n,w^n)$ is a minimizer of problem  \eqref{Pb:mw2} associated with $H^n$, $f^n$, $m_0^n$ and $\phi_T^n$, while the pair $(\phi^n,\alpha^n)$ is a minimizer of problem \eqref{PB:dual-relaxed} associated with the same data. Using the estimates established for the proof of Proposition \ref{prop:valeursegales}, we have 
\be\label{boundmnwn}
\|m^n\|_{L^q}+ \|w^n\|_{L^{\frac{r'q}{r'+q-1}}} \leq C. 
\ee
Standard $\Gamma-$convergence arguments then show that $(m^n,w^n)$ converge in $L^q\times L^{\frac{r'q}{r'+q-1}}$ to the unique minimum of the problem \eqref{Pb:mw2} associated with $H$, $f$, $m_0$ and $\phi_T$. 
 
Estimate \eqref{boundmnwn} and the growth condition \eqref{Hypf} on $f$ imply that the sequence $(\alpha^n=f(\cdot,m^n))$  in $L^p$ to $\alpha:= f(\cdot, m)$. Lemma \ref{lem:estiPhiT} then gives an upper bound  for the $\phi^n$, while the additional condition \eqref{CondSup} provides a lower bound. Arguing as in the proof of Proposition \ref{Prop:existence}, one can show that inequality \eqref{eq:distrib} combined with the $L^\infty$ bound on $\phi^n$ provides a bound on $\|D\phi^n\|_{L^r}$ and on $\|\phi^n\|_{BV}$. Finally, Lemma \ref{CShold} provides a uniform Hölder continuity of $\phi^n$ in any compact subset of $[0,T)\times \T^d$. Hence $(\phi^n)$ converges, up to a subsequence, locally uniformly to a map $\phi$.   Then, as in the proof of Proposition \ref{Prop:existence}, the pair $(\phi,\alpha)$ belongs to $\mathcal K$ and is a minimizer of \eqref{PB:dual-relaxed}. 

As $(m,w)$ solves \eqref{Pb:mw2} while $(\phi, \alpha)$ is a solution of  \eqref{PB:dual-relaxed}, the first part of Theorem \ref{theo:main} implies that the pair $(\phi,m)$ is a weak solution of \eqref{MFG}. Since the $\phi^n$ satisfy  the additional condition \eqref{CondSup}, so does $\phi$. Therefore $(\phi,m)$ is the unique weak solution of \eqref{MFG} which satisfies \eqref{CondSup}. This shows that the full sequence $(m^n,\phi^n)$ converges to $(m,\phi)$. 
\end{proof}


\section{Application to differential games with  finitely many players}\label{sec:jeufini}

We now explain  how the solution of the mean field game system can be used to derive approximate Nash equilibria for differential games with finitely many players.

\subsection{Model}

In order to define the differential game, we introduce (or recall) few notations. We let $N$ be the number of players. As before we denote by $L$ the convex conjugate of the map $p\to H(x,-p)$, i.e., $L(x,\xi)=H^*(x,-\xi)$. The map $L$ will be the uncoupled part of the cost of a single player. The coupled part will be given by a regularization of the coupling $f$. For this, let us fix a smooth, symmetric and  nonnegative regularization kernel $\xi:\R^d\to \R$ and let us set, for $\delta>0$, $\xi_\delta(x)= \frac{1}{\delta^d}\xi(\frac{x}{\delta^d})$. For $\delta, \sigma>0$,  the regularized coupling is the map $f^{\delta,\sigma}:\T^d\times P(\T^d)\to \R$ defined by
$$
f^{\delta, \sigma}(x,\mu)= (f^\delta(\cdot,\mu)\star \xi^\sigma) (x) \; {\rm where }\; f^\delta(x,\mu) = f(x, \xi^\delta\star \mu(x)). 
$$
The idea is that the parameter $\delta$ allows to give a meaning to the expression $f(x,\mu)$ when $\mu$ is a singular measure, while the second regularization in $\sigma$ ensures a space regularity of the resulting map when $\delta$ is small. When $\mu$ is in $L^1(\T^d)$, we set (with a slight abuse of notation)  $f^{\sigma}=f^{0,\sigma}$. 
We often use the above definition for empirical measures of the form $\mu= \frac{1}{N-1}\sum_{j\neq i}^N\delta_{x^j}$ (where $i\in 1, \dots, N$ and $x^j\in \T^d$ for $j=1,\dots, N$ for $j\neq i$):  then
$$
f^\delta\left(x,\frac{1}{N-1}\sum_{j\neq i}\delta_{x^j}\right) = f\left(x, \frac{1}{N-1}\sum_{j\neq i}\xi^\delta(x-x^j)\right)
$$
while 
$$
f^{\delta,\sigma}\left(x,\frac{1}{N-1}\sum_{j\neq i}\delta_{x^j}\right) =\int_{\R^d} \xi^\sigma(x-y)f\left(y, \frac{1}{N-1}\sum_{j\neq i}\xi^\delta(y-x^j)\right)dy.
$$

Let us start with the model.  Recall that $N$ is the number of players. Player $i$ has a current position denoted by $\gamma^i(t)$ and controls its own velocity $\dot \gamma^i(t)$.  At time $0$, the initial position $x^i_0$ of player $i$ (where $i=1, \dots, N$) is  chosen randomly  with probability $m_0$. So the trajectory $ \gamma^i$ satisfies $\gamma^i(0)= x^i_0$.  We assume that the random variables  $x^1_0, \dots, x^N_0$ are independent. If the players play a family of  trajectories $\gamma^1, \dots, \gamma^N$, the cost of player $i$ is given by 
\be\label{defsimplecost}
J_i^N (\gamma^1, \dots, \gamma^N) = \int_0^T \left( L(\gamma^i(s), \dot \gamma^i(s)) + f^{\delta,\sigma}\left(\gamma^i(s), \frac{1}{N-1}\sum_{j\neq i} \delta_{\gamma^j(s)}\right)\right) ds + \phi_T(\gamma^i(T))
\ee
Players can play random strategies with delay. To define this notion, let us fix a standard probability space $(\Omega, {\mathcal F}, \P)$ (in practice, we choose $\Omega=[0,1]$, ${\mathcal F}$ is the Borel $\sigma-$algebra and $\P$ is the Lebesgue measure). A strategy for player $i$ is a Borel measurable map $\beta^i:\Omega\times \T^d\times \Gamma^{N-1}\to \Gamma$ such that, 
\begin{itemize}
\item[(i)] for any $(\omega, x, (\gamma^{j})_{j\neq i}) \in \Omega\times \T^d\times \Gamma^{N-1}$, $\beta^i(\omega, x, (\gamma^{j})_{j\neq i})(0)=x$, 

\item[(ii)] there is a delay $\tau>0$ with the property that, for any $(\omega, x)\in \Omega\times \T^d$ and any $(\gamma^{1,j})_{j\neq i}$, $(\gamma^{2,j})_{j\neq i}$ with  $\gamma^{1,j}(s)= \gamma^{2,j}(s)$ for any $j\neq i$ and $s\in [0,t]$, the responses $\beta^i(\omega, x, (\gamma^{1,j})_{j\neq i})$ and 
$\beta^i(\omega, x,(\gamma^{2,j})_{j\neq i})$ coincide on $[0,t+\tau]$.  
\end{itemize}
The interpretation is that player $i$ observes his initial position $x$ and the other players' trajectories (in a nonanticipative way) and answers a random trajectory starting at $x$;  the parameter $\Omega$ formalizes the random device, as in Aumann \cite{Au61}. Moreover there is a small delay (the quantity $\tau$) between the observation and the reaction. This delay can be arbitrarily small. 

Given $N$ (independent) strategies $(\beta^1, \dots, \beta^N)$ and a family of initial conditions $(x^1_0, \dots, x^N_0)$, one can associate a unique family of Borel measurable maps ${\bf \gamma}^i: \Omega^N\times (\T^d)^N \to \Gamma$ (for $i=1, \dots, N$) which satisfies, 
\be\label{equil}
\beta^i(\omega^i, x_0^i, (\gamma^j(\omega,x))_{j\neq i})= \gamma^i(\omega,x) \qquad  \mbox{\rm for any $i=1, \dots, N$},
\ee
where $x_0=(x_0^1, \dots, x_0^N)$ and $\omega=(\omega^1, \dots,\omega^N)$: this is just a consequence of the delay of the strategies (see, e.g.,  \cite{c1}, \cite{CaQu}). 

Recalling that the initial conditions are chosen randomly with probability $m_0$, we are finally ready to define the cost, for player $i$, of a family of strategies $(\beta^1, \dots, \beta^N)$: it is given by 
$$
{\bf J^N_i}(\beta^1, \dots, \beta^N)= \int_{\Omega^N\times (\T^d)^N} J^N_i(\gamma^1(\omega,x), \dots, \gamma^N(\omega,x))
\prod_{j=1}^N \P(d\omega_j)m_0(dx_0^j)
$$
where $J^N_i$ is defined in \eqref{defsimplecost} and where the family of trajectories $(\gamma^1(\omega,x), \dots, \gamma^N(\omega,x))$ is characterized by the fixed point relation \eqref{equil}. In order to single out the behavior of player $i$, we often write ${\bf J^N_i}((\beta^j)_{j\neq i}, \beta^i)$ for ${\bf J^N_i}(\beta^1, \dots, \beta^N)$.

Let us finally give examples of strategies for player $i$: an elementary one is given by a Borel measurable $\beta^i: \T^d\to \Gamma$: such a strategy is deterministic (it does not depend on $\Omega$) and open-loop (it does not depend on the other players' actions). It associates with any initial condition $x^i_0\in\T^d$ a trajectory $\beta^i(x^i_0)$ starting at $x_0^i$. We will be particularly interested in random open-loop strategies $\beta^i: \Omega\times \T^d \to \Gamma$. They are now random (they depend on $\Omega$) but are still open-loop (no dependence with respect to the other players' trajectories). In fact, under few restriction, there is a one-to-one correspondence between these strategies and the probability measures on curves introduced in subsection  \ref{subsec:uniq}. Indeed, let $\beta^i$ be as above. Since the initial position $x_0^i$ of player $i$ is chosen randomly with probability $m_0$, one can associate  with $\beta^i$ the measure $\eta$ on $\Gamma$ defined by the equality
\be\label{defbetai}
\int_{ \T^d\times \Gamma} G(\gamma) d\eta(\gamma) = \int_{\Omega\times \T^d} G(\beta^i(\omega,x))d\P(\omega)m_0(dx),
\ee
for any continuous and bounded map $G:\Gamma\to \R$. 
Note that, by definition, $e_0\sharp \eta= m_0$. If we further assume that 
$$
\int_{\Gamma} \int_0^T \left|\frac{d}{dt} \beta^i(\omega,x)(s)\right|^{r'} \ ds m_0(dx)<+\infty,
$$
and that $m'(t):= e_t\sharp \eta$ is absolutely continuous for any $t\in [0,T]$ (the density being denoted by $m'(t,\cdot)$), with $m'\in  L^q((0,T)\times\T^d)$, then $\eta$ belongs to the set $\tilde M(\Gamma)$ defined in subsection \ref{subsec:uniq}.  Conversely, let $\eta \in \tilde M(\Gamma)$ and assume that $m_0\sharp \eta=m_0$. Then by desintegration there exists a Borel measurable family of probabilities $(\eta_x)_{x\in \T^d}$ on $\Gamma$ such that $\ds \eta(d\gamma)= \int_{\T^d}\eta_x(d\gamma) m_0(dx)$ and, for $m_0-$a.e. $x\in \T^d$, $\eta_x-$almost any trajectory $\gamma$ starts at $x$. Using the Blackwell-Dubins Theorem \cite{BD83},  one can represent the family of probability measures $(\eta_x)$ by a single map, which is exactly a random open-loop strategy: there exists a measurable map $\beta^i: \Omega\times \T^d \to \Gamma$ satisfying the relation \eqref{defbetai} for any continuous and bounded map $G:\T^d\times \Gamma\to \R$.  To simplify notations, we will say that  $\eta$ itself  is a random, open-loop strategy.

\subsection{Existence of approximate Nash equilibria in open-loop strategies}

Let $(m,\phi)$ be the unique  weak solution of the mean field game system \eqref{MFG} such that the additional condition \eqref{CondSup} holds. Let  $\bar \eta\in \tilde M(\Gamma)$ be such that 
\begin{itemize}
\item[(C1)] $m(t)= e_t\sharp \bar \eta$ for any $t\in [0,T]$, 

\item[(C2)] the following equality holds: 
$$\begin{array}{rl}
\ds \int_{\T^d} \phi(0,x)m(0,x) dx \; = & \ds \int_{\T^d} \phi_T(x)m(T,x)+ \int_\Gamma\int_0^T L(\gamma(s),\dot\gamma(s)) dsd\bar \eta(\gamma) \\
& \ds \qquad \qquad +\int_{\T^d}\int_0^T f(x,m(s,x)) m(s,x)dsdx 
\end{array}
$$
\end{itemize} 
The existence of such a measure $\bar \eta$ is guaranteed by Lemma \ref{lem:equalpsi}. 

\begin{Theorem}\label{theo:Njoueurs}  Assume that $f$ is uniformly Lipschitz continuous with respect to the second variable. For any $\ep>0$ there exist $N_0$, $\delta,\sigma>0$ such that, if $N\geq N_0$, the family of open-loop strategies $(\bar \eta, \dots, \bar \eta)$ is an approximate Nash equilibrium for the game: namely,  for any strategy $\beta^i$ of player $i$, 
$$
{\bf J^N_i}((\bar \eta)_{j\neq i}, \beta^i) \geq {\bf J^N_i}((\bar \eta)_{j\neq i}, \bar \eta)- \ep. 
$$
Moreover, 
\be\label{finalopeitmalcost}
\left|{\bf J^N_i}((\bar \eta)_{j\neq i}, \bar \eta)- \int_{\T^d} \phi(0,x)m_0(x)dx\right| \leq \ep.
\ee
\end{Theorem}

\begin{Remarks}\end{Remarks}
\begin{enumerate}
\item The key point in the above result is that no player can improve his payoff in a substantial way by changing its strategy, even by observing the other players.

\item Inequality \eqref{finalopeitmalcost} says that the average optimal cost of a player is approximately given by the quantity $\ds \int_{\T^d} \phi(0,x)m_0(x)dx$.

\item The measure $\bar \eta$ satisfying the above conditions (C1) and (C2) need not be unique. However, given, for each $i=1, \dots,N$, a measure $\bar \eta^i$ satisfying conditions (C1) and (C2), one can prove exactly in the same way that the family $(\bar \eta^1, \dots, \bar \eta^N)$ is an approximate Nash equilibrium: in particular, players do not need to coordinate to choose the open loop strategy. 

\item It would be much more natural {\it not to assume} that the initial conditions of the players are chosen i.i.d. according to the measure $m_0$, but just to suppose that the empirical distribution $\frac{1}{N} \sum_{i=1}^N \delta_{x_0^i}$ of the fixed initial positions $(x_0^1, \dots, x_0^N)$ of the players is close to $m_0$. However we do not know how to handle this problem. 
\end{enumerate}

\begin{proof}[Proof of Theorem  \ref{theo:Njoueurs}.] Before starting the proof, we need to fix notations. Let us fix a strategy $\beta^i$ for player $i$ and assume that the other players play the open-loop strategy $\bar \eta$. Recall that  one can associate with $\bar \eta$ a genuine strategy $\beta^j:\Omega\times \T^d\to\T^d$ such that \eqref{defbetai} holds (with $\eta$ replaced by $\bar \eta$ and $\beta^i$ replaced by $\beta^j$). Let ${\bf \gamma}^j: \Omega^N\times (\T^d)^N \to \Gamma$ (for $j=1, \dots, N$) be  the family Borel measurable maps  given by the fixed point relation \eqref{equil} holds. As  (for $j\neq i$) the strategy $\beta^j$ does not depend on the other players' behavior,  we have ${\bf \gamma}^j(\omega, x_0)= \beta^j(\omega^j,x^j_0)$. On another hand, $\gamma^i(\omega,x_0)$ a priori depends on all trajectories $({\bf \gamma}^j)_{j\neq i}$.
Accordingly we can rewrite the cost ${\bf J^N_i}((\bar \eta)_{j\neq i}, \beta^i)$ as 
$$
{\bf J^N_i}((\bar \eta)_{j\neq i}, \beta^i)= \int_{\Omega\times \T^d\times \Gamma^{N-1}}
J^i_N\left((\gamma^j)_{j\neq i}, \beta^i (\omega^i,x_0^i, (\gamma^j)_{j\neq i})\right) \P(d\omega^i)m_0(dx_0^i)\prod_{j\neq i} \bar \eta(d\gamma^j)
$$
To simplify notations we will simply write $\ds  \gamma^{x_0^i,\beta^i}$ for $ \beta^i (\omega^i,x_0^i, (\gamma^j)_{j\neq i})$ but keep in mind that $\ds  \gamma^{x_0^i,\beta^i}$ still depends on $\omega^i$ and on the $(\gamma^j)_{j\neq i}$.\\

Next we establish preliminary estimates. 
By definition of the open-loop strategies $\bar \eta$, the  $\gamma^j(t)$ (for $j\neq i$) are iid random variables of law $e_s\sharp \bar \eta=m(t)$. Following Section 10 in \cite{RR98}, we have therefore
$$
\int_{\Gamma^{N-1}} {\bf W}_2^2\left( \frac{1}{N-1}\sum_{j\neq i} \delta_{\gamma^j(s)}, m(t)\right) \prod_{j\neq i} d\bar \eta(\gamma^j) \leq  CN^{-2/(d+4)}
$$
(where ${\bf W}_2$ is the $2-$Wasserstein distance). 
As the map $m\to f(x,m)$ is uniformly Lipschitz continuous, the map $f^{\delta,\sigma}$ satisfies, for any $x^i\in \T^d$ and any $\mu,\nu\in P(\T^d)$, 
$$
\left| f^{\delta,\sigma}(x^i, \mu)-f^{\delta,\sigma}(x^i,\nu)\right| \leq C \left|\xi^\delta\star(\mu-\nu)(x)\right| \leq C {\rm Lip}(\xi^\delta) {\bf W}_2(\mu,\nu)\leq C \delta^{-(d+1)}{\bf W}_2(\mu,\nu). 
$$
Therefore
\be\label{intintint}
\left| \int_{\Gamma^{N-1}} f^{\delta,\sigma}\left(x^i, \frac{1}{N-1}\sum_{j\neq i} \delta_{\gamma^j(s)}\right) \prod_{j\neq i} d\bar \eta(\gamma^j) - f^{\delta,\sigma}(x^i, m(s))\right|\leq C\delta^{-(d+1)}N^{-2/(d+4)}  
\ee
Note also that, for any $x\in \T^d$ and $t\in [0,T]$,  
\be\label{intintint2}
\begin{array}{rl}
\ds \left| f^{\delta,\sigma}(x,m(t))-f^{\sigma}(x,m(t))\right| \; \leq & \ds  \int_{\R^d} \xi^\sigma(x-y)\left| f(y, (\xi^\delta\star m(t))(y))-f(y,m(t,y))\right| dy \\
\leq & \ds C \int_{\R^d} \xi^\sigma(x-y)\left| (\xi^\delta\star m(t))(y)-m(t,y)\right| dy \\
\leq & \ds C \|\xi^\sigma\|_{L^p}\left\| \xi^\delta\star m(t)-m(t,\cdot)\right\|_{L^q} =: C_\sigma(t,\delta)
\end{array}
\ee
where $C_\sigma(\cdot,\delta)$ tends to $0$ in $L^q((0,T))$ as $\delta\to 0$ because  $\xi^\delta\star m$ converges to $m$ in $L^q((0,T)\times \T^d)$. 
We set 
$$
C_\sigma (\delta):= \int_0^T C_\sigma(t,\delta)dt 
$$
and keep in mind that $C_\sigma (\delta)\to 0$ as $\delta\to 0$.

We are now ready to start the proof of the theorem. Let us estimate the cost of player $i$ when he plays the strategy $\beta^i$: by \eqref{intintint}
we have
$$
\begin{array}{rl}
\ds {\bf J_i^N}((\bar \eta)_{j\neq i}, \beta^i)\; =  & \ds \int_{\Omega\times \T^d\times \Gamma^{N-1}} \left[\int_0^T \left(L\left(\gamma^{ x^i_0, \beta^i}_s, \dot\gamma^{ x^i_0, \beta^i}_s\right) + f^{\delta,\sigma}(\gamma^{ x^i_0, \beta^i}_s,\frac{1}{N-1}\sum_{j\neq i} \delta_{\gamma^j(s)})\right)ds \right.\\
& \ds \qquad \qquad \qquad \qquad \left.+ \phi_T(\gamma^{ x^i_0, \beta^i}_T) \right] d\P(\omega^i)dm_0( x^i_0)\prod_{j\neq i} d\bar \eta(\gamma^j) \\
\geq & \ds \int_{\Omega\times \T^d\times \Gamma^{N-1}} \left[\int_0^T \left(L\left(\gamma^{ x^i_0, \beta^i}_s, \dot\gamma^{ x^i_0, \beta^i}_s\right) + f^{\delta,\sigma}(\gamma^{ x^i_0, \beta^i}_s,m(s))\right)ds 
\right.\\
& \ds \qquad \qquad \qquad \left.
+ \phi_T(\gamma^{x^i_0, \beta^i}_T) \right] d\P(\omega^i)dm_0( x^i_0)\prod_{j\neq i} d\bar \eta(\gamma^j)- C\delta^{-(d+1)}N^{-2/(d+4)}.
\end{array}
$$
Note that the only dependence  with respect to the $(\gamma^j)_{j\neq i}$ of the integrand in the above expression is just through $\gamma^{x^i_0, \beta^i}$. Therefore it convenient to introduce  the probability measure $\eta$ on $\Gamma$ as 
$$
\int_{\Gamma} G(\gamma)d\eta(\gamma)= \int_{\Omega\times \T^d\times \Gamma^{N-1}} G\left(\gamma^{x^i_0, \beta^i}\right) d\P(\omega^i)dm_0( x^i_0)\prod_{j\neq i} d\bar \eta(\gamma^j)
$$
for any continuous and bounded map $G$ on $\Gamma$. Then 
\be\label{plugging}
\begin{array}{rl}
\ds {\bf J_i^N}((\bar \eta)_{j\neq i}, \beta^i)\; \geq   & \ds \int_{\Gamma} \left[\int_0^T \left(L\left(\gamma(s), \dot\gamma(s)\right) + f^{\delta,\sigma}(\gamma(s),m(s))\right)ds 
\right.\\
& \ds \qquad \qquad \qquad \qquad 
+ \phi_T(\gamma(T)) \Big] d\eta(\gamma) - C\delta^{-(d+1)}N^{-2/(d+4)}\\
\geq & \ds  \int_{\Gamma} \left[\int_0^T \left(L\left(\gamma(s), \dot\gamma(s)\right) + f^{\sigma}(\gamma(s),m(s))\right)ds 
\right.\\
& \ds \qquad \qquad \qquad \qquad 
+ \phi_T(\gamma(T)) \Big] d\eta(\gamma) - C\delta^{-(d+1)}N^{-2/(d+4)}-C_\sigma (\delta), 
\end{array}
\ee
where the last inequality comes from \eqref{intintint2}.
Let $\phi^\sigma$ be the unique continuous viscosity solution of the Hamilton-Jacobi equation (with a  time-measurable dependent Hamiltonian, see \cite{LP87})
\be\label{eqphisigma}
\left\{\begin{array}{l}
-\partial_t\phi^\sigma+H(x,D\phi^\sigma) = f^\sigma(x,m(t)) \;\; {\rm in}\; (0,T)\times \T^d\\
\phi^\sigma(T,x)= \phi_T(x)\; {\rm in}\;\; \T^d
\end{array}\right.
\ee
By definition of $f^\sigma$ and following the proof of Proposition \ref{prop:valeursegales} the map $\psi^\sigma:=\phi\star \xi^\sigma$ is a subsolution of 
$$
\left\{\begin{array}{l}
-\partial_t\psi^\sigma+H(x,D\psi^\sigma) \leq  f^\sigma(x,m(t))+C  \sigma^{1-(d+1)\theta/r}(1+ \|D\phi\|_r^\theta)\;\;  {\rm in}\;\; (0,T)\times \T^d)\\
\psi^\sigma(T,x)\leq \phi_T(x)+C \sigma\;\; {\rm in}\;  \T^d
\end{array}\right.
$$
By comparison, we get
$$
\psi^\sigma= \phi\star \xi^\sigma\leq \phi^\sigma +CT \sigma+ C  \sigma^{1-(d+1)\theta/r}(1+ \|D\phi\|_r^\theta)\leq \phi^\sigma +C \sigma^{1-(d+1)\theta/r}.
$$
In particular, using the continuity of $\phi(0,\cdot)$, we obtain, for $\sigma$ small enough, 
\be\label{lowerboundphisigma}
\phi(0,\cdot) \leq \phi^\sigma(0,\cdot) +\frac{\ep}{10}.
\ee
Recalling that $\phi^\sigma$ solves \eqref{eqphisigma} and that $e_0\sharp \eta=m_0$, we can follow the computation of Lemma \ref{lem:t1t2} with $t_1=0$ and $t_2=T$  to get
$$
\int_{\T^d} \phi^\sigma(0,x)m_0(x) \leq \int_{\Gamma} \left[\int_0^T \left(L\left(\gamma(s), \dot\gamma(s)\right) + f^{\sigma}(\gamma(s),m(s))\right)ds + \phi_T(\gamma(T)) \right] d\eta(\gamma).
$$
Plugging \eqref{lowerboundphisigma} and the last inequality into \eqref{plugging} gives 
\be\label{estipourbetai}
\begin{array}{rl}
\ds {\bf J_i^N}((\bar \eta)_{j\neq i}, \beta^i)\; \geq  & \ds \int_{\T^d} \phi(0,x)m_0(x) - C\delta^{-(d+1)}N^{-2/(d+4)}-C_\sigma (\delta) -\frac{\ep}{10}.
\end{array}
\ee
This inequality holds for any strategy $\beta^i$ and for $\sigma$ so small that \eqref{lowerboundphisigma} holds (note that this condition does not involve $\beta^i$). 
We now assume that player $i$ plays the open-loop control $\bar \eta$. Arguing as for the proof of \eqref{plugging}, we have
$$
\begin{array}{rl}
\ds {\bf J_i^N}((\bar \eta)_{j\neq i}, \bar \eta)\; \leq   & \ds   \int_{\Gamma} \left[\int_0^T \left(L\left(\gamma(s), \dot\gamma(s)\right) + f^{\sigma}(\gamma(s),m(s))\right)ds 
\right.\\
& \ds \qquad \qquad \qquad \qquad 
+ \phi_T(\gamma(T)) \Big] d\bar \eta(\gamma) + C\delta^{-(d+1)}N^{-2/(d+4)}+C_\sigma (\delta). 
\end{array}
$$
By assumption (C1) on $\bar \eta$ we have
$$
\int_{\Gamma}\phi_T(\gamma(T))  d\bar \eta(\gamma) = \int_{\T^d} \phi_T(x)m(T,x)dx
$$
and
$$
\begin{array}{rl}
\ds \int_{\Gamma} \int_0^T f^{\sigma}(\gamma(s),m(s))ds d\bar \eta(\gamma)= & \ds  \int_0^T \int_{\T^d} f^{\sigma}(y,m(s))m(s,y)dyds.
\end{array}$$
If we choose $\sigma$ small enough, we have
$$
\begin{array}{rl}
\ds \int_{\Gamma} \int_0^T f^{\sigma}(\gamma(s),m(s))ds d\bar \eta(\gamma)
  \leq & \ds 
 \int_0^T \int_{\T^d} f(y,m(s,y))m(s,y)dyds  +\frac{\ep}{10}.
\end{array}$$
So, using condition (C2) on $\bar \eta$ 
we obtain
\be\label{estipourbareta}
\begin{array}{rl}
\ds {\bf J_i^N}((\bar \eta)_{j\neq i}, \bar \eta)\; \leq   & \ds   \int_{\Gamma}\int_0^T L\left(\gamma(s), \dot\gamma(s)\right) ds d\bar \eta(\gamma)+
\int_0^T \int_{\T^d} f(y,m(s,y))m(s,y)dyds \\
& \ds \qquad \qquad \qquad 
 +\int_{\T^d} \phi_T(y) m(T,y)dy  + C\delta^{-(d+1)}N^{-2/(d+4)}+C_\sigma (\delta) +\frac{\ep}{10}\\
 \leq & \ds  \int_{\T^d} \phi(0,x)m_0(x)dx+ C\delta^{-(d+1)}N^{-2/(d+4)}+C_\sigma (\delta) +\frac{\ep}{10}
\end{array}
\ee
For $\sigma$ small as above, let us choose $\delta$ so small enough that $C_\sigma (\delta)\leq \frac{\ep}{10}$ and $N_0$ so large that 
$C\delta^{-(d+1)}N_0^{-2/(d+4)}\leq \frac{\ep}{10}$. Then, for $N\geq N_0$, we have by \eqref{estipourbetai} and \eqref{estipourbareta}
$$
\begin{array}{rl}
\ds {\bf J_i^N}((\bar \eta)_{j\neq i}, \beta^i)\; \geq  & \ds \int_{\T^d} \phi(0,x)m_0(x) -\frac\ep2
\end{array}
$$
and 
$$
{\bf J_i^N}((\bar \eta)_{j\neq i}, \bar \eta) \leq  \int_{\T^d} \phi(0,x)m_0(x) +\frac\ep2, 
$$
which completes the proof of the theorem. 
\end{proof}

\section{Comparison principle  and link with a time-space elliptic equation}\label{sec:visco}

In \cite{LLperso}, Lions shows that classical solutions of the MFG system \eqref{MFG} satisfy a comparison principle. Moreover, he explains that the MFG system can be reduced to an elliptic equation in $[0,T]\times \T^d$. We show here that these two properties extend---not very surprizingly---to our weak solutions of \eqref{MFG}. 

\subsection{Comparison principle}

\begin{Proposition} Let us fix $m_0$ but assume that we are given two terminal conditions $\phi_T^1$ and $\phi_T^2$, with $\phi_T^1\leq \phi_T^2$. Let $(\phi^1,m^1)$ and $(\phi^2,m^2)$ be the associated solutions of \eqref{MFG}, with the additional condition \eqref{CondSup}. Then $\phi^1\leq \phi^2$ on $[0,T]\times \T^d$.
\end{Proposition}

A possible application of the above proposition is the following: assume that $H$ and $f$ are independent of $x$ and that $m_0=1$ a.e.. Let $(\phi,m)$ be the weak solution to \eqref{MFG} which satisfies \eqref{CondSup}. Then one easily checks that $\phi$ is Lipschitz continuous in space. 

\begin{proof} Let us set $\alpha^i= f(\cdot,m^i)$ for $i=1,2$. We use the fact that, $(\phi^1,\alpha^1)$ and $(\phi^2,\alpha^2)$ are minimizers of \eqref{PB:dual-relaxed} with terminal conditions  $\phi_T^1$ and $\phi_T^2$ respectively. Recall that the additional condition  \eqref{CondSup} ensures that $\phi^1$ (resp. $\phi^2$) is the maximal subsolution of \eqref{subHJ} with terminal condition $\phi_T^1$ (resp. $\phi^T_2$). 

We argue by contradiction, assuming that $\max(\phi^1-\phi^2)>0$. Let ${\mathcal O}=\{\phi^1>\phi^2\}\subset [0,T)\times \T^d$ and 
$$
I_1=\int\int_{{\mathcal O}} F^*(x,\alpha^1)dxdt-   \int_{{\mathcal O}\cap \{t=0\}} \phi^1(0)m_0 
$$
and
$$
 I_2=\int\int_{{\mathcal O}} F^*(x,\alpha^2)dxdt -   \int_{{\mathcal O}\cap \{t=0\}} \phi^2(0)m_0 \;.
$$
If $I_1\leq I_2$, then we set $\phi=\max\{\phi^1,\phi^2\}$ and $\alpha=\alpha^1{\bf 1}_{{\mathcal O}}+ \alpha^2{\bf 1}_{{\mathcal O}^c}$. Then we have 
$(\psi, \alpha)\in {\mathcal K}$ with $\psi(T,\cdot)=\phi^2_T$ and 
$$
\int_0^T\int_{\T^d}  F^*(x,\alpha) - \int_{\T^d} \psi(0)m_0
\leq
\int_0^T\int_{\T^d}  F^*(x,\alpha^2) - \int_{\T^d} \phi^2(0) m_0\;.
$$
In particular, $(\psi,\alpha)$ is another solution of \eqref{PB:dual-relaxed}, with $\psi\geq \phi^2$. This contradicts the maximality of $\phi^2$.

If, on the contrary, $I_1>I_2$, then we can argue in a symmetric way by comparing $(\psi,\alpha)$ (where $\psi= \min\{\phi^1,\phi^2\}$ and $\alpha=\alpha^1{\bf 1}_{{\mathcal O}^c}+ \alpha^2{\bf 1}_{{\mathcal O}}$) with $(\phi^1,\alpha^1)$. We get now a contradiction because 
$$
\int_0^T\int_{\T^d}  F^*(x,\alpha) - \int_{\T^d} \psi(0)m_0
<
\int_0^T\int_{\T^d}  F^*(x,\alpha^1) - \int_{\T^d} \phi^1(0) m_0\;.
$$
\end{proof}

\subsection{The MFG system as a time-space elliptic equation}

We now show that, if $(\phi,m)$ is a weak solution to the MFG system, then $\phi$ is also a viscosity solution of a degenerate elliptic equation in time-space. 

We work under the additional assumptions
\be\label{HypSup}
\mbox{\rm the maps $H$ and $F^*$ are of class ${\mathcal C}^2$ in $\T^d\times \R^d$ and $\T^d\times (0,+\infty)$ respectively}
\ee
and 
\be\label{HypSup2}
\mbox{\rm $F^*_{\alpha\alpha}(x,\alpha)>0$ for $\alpha>0$.}
\ee

\begin{Proposition} Assume that  $(\phi,m)$ is the weak solution of \eqref{MFG} for which \eqref{CondSup} holds. Then $\phi$ is a viscosity solution of the second order elliptic equation
\be\label{HJ}
\left\{\begin{array}{l}
\ds \min\left\{ \ {\mathcal G}\left(x, \partial_t\phi, D\phi, \partial_{tt}\phi, D\partial_t\phi, D^2\phi\right)\ ;\
-\partial_t\phi+H(x,D\phi))\ \right\}=0\qquad {\rm in}\; (0,T)\times \T^d\\
\ds \phi(T,\cdot) =\phi_T \qquad {\rm in}\; \T^d\\
\ds -\partial_t\phi+H(x,D\phi)=f(m_0) \qquad {\rm in}\; \T^d
\end{array}\right.
\ee
where, for any $(x,p_t,p_x, a,b,C)\in \T^d\times \R\times \R^d\times \R\times \R^d\times \R^{d\times d}$ with $-p_t+H(x,p_x)> 0$
$$
\begin{array}{l}
\ds {\mathcal G}(x,p_t,p_x, a,b,C)\\
\;  =  \ds  F^{*}_{\alpha,\alpha}\left[ -a +2\lg H_p, b\rg -\lg C H_p, H_p\rg -\lg H_p, H_x\rg\right]
 - \lg F^*_{x,\alpha}, H_p\rg - F^*_\alpha\left[ Tr(H_{x,p})+Tr(H_{pp}C)\right]\\
 \\
\;  =  \ds -Tr\left( {\mathcal A}(x,p_t,p_x)\left(\begin{array}{cc} a & b^T\\b & C\end{array}\right)\right) 
- F^{*}_{\alpha,\alpha}\lg H_p, H_x\rg
- \lg F^*_{x,\alpha}, H_p\rg - F^*_\alpha Tr(H_{x,p})
\end{array}
$$
and where 
$$
{\mathcal A}(x,p_t,p_x)= F^*_{\alpha,\alpha}\left(\begin{array}{cc} 1 & -H_p^T \\ -H_p &  H_p\otimes H_p  \end{array}\right)
+  F^*_{\alpha}\left(\begin{array}{cc} 0 & 0 \\ 0 &  H_{pp}  \end{array}\right).
$$
In the above  equations we have systematically set 
$$
H= H(x,p_x), \; H_p= \frac{\partial H}{\partial p}(x,p_x), \; H_x= \frac{\partial H}{\partial x}(x,p_x)
$$
$$
F^{*}=  F^*(x,-p_t+H(x,p_x)), \; F^{*}_\alpha= \frac{\partial F^*}{\partial \alpha}(x,-p_t+H(x,p_x)), \; \dots
$$
If $-p_t+H(x,p_x)\leq 0$, we simply set $\ds {\mathcal G}(x,p_t,p_x, a,b,C)=0$. 
\end{Proposition}

\begin{Remarks} {\rm $\;$
\begin{enumerate}
\item The boundary condition at time $t=0$ has to be understood in the viscosity sense. 

\item Equation \eqref{HJ} is a quasilinear (degenerate) elliptic equation in space-time. 
\end{enumerate}
}\end{Remarks}

\begin{proof} The proof follows  standard tricks. Let us first check that $\phi$ is a subsolution: let $\xi$ be a smooth test function such that $\xi\geq \phi$ with an equality only at $(t_0,x_0)\in (0,T)\times \T^d$. We have to prove that 
 \be\label{kjhcvbss}
\min\left\{ \ {\mathcal G}\left(x_0, \partial_t\xi, D\xi, \partial_{tt}\xi, D\partial_t\xi, D^2\xi\right)\ ;\
-\partial_t\xi+H(x,D\xi))\ \right\}\leq 0\qquad {\rm at }\; (t_0,x_0).
\ee
If $-\partial_t\xi(t_0,x_0)+H(x,D\xi(t_0,x_0)))\leq 0$, then the result holds. Let us assume that $-\partial_t\xi(t_0,x_0)+H(x,D\xi(t_0,x_0)))> 0$.
For any $\ep>0$, we set ${\mathcal O}_\ep=\{\xi-\ep<\phi\}$, 
$$
\phi_\ep= \phi\wedge (\xi-\ep), \; \alpha_\ep = (-\partial_t \xi+H(x,D\xi)){\rm 1}_{{\mathcal O}_\ep}+ \alpha {\rm 1}_{{\mathcal O}_\ep^c}.
$$
We can choose $\ep>0$ so small that  $-\partial_t\xi+H(x,D\xi))> 0$ in ${\mathcal O}_\ep$. Then  the pair $(\phi_\ep,\alpha_\ep)$ is still admissible and by optimality of $(\phi,\alpha)$ we have, for $\ep$ small enough so that ${\mathcal O}_\ep\subset (0,T)\times \T^d$,  
$$
\iint_{{\mathcal O}_\ep}F^*(x, -\partial_t \xi+H(x,D\xi))  \geq\iint_{{\mathcal O}_\ep}F^*(x,\alpha) 
$$
where, as $F^*$ is convex, 
$$
\iint_{{\mathcal O}_\ep}F^*(x,\alpha) 
\geq \iint_{{\mathcal O}_\ep}F^*(x, -\partial_t \xi+H(x,D\xi))+ F^*_\alpha(x, -\partial_t \xi+H(x,D\xi))[\alpha-(-\partial_t \xi+H(x,D\xi))]
$$
As $F^*$ is nondecreasing in the second variable  and $\alpha \geq -\partial_t\phi+H(x,D\phi)$ in the sense of measure
we have (writing $F^*_\alpha$ for $F^*_\alpha(x, -\partial_t \xi+H(x,D\xi))$)
$$
0\geq  \iint_{{\mathcal O}_\ep} F^*_\alpha\left[ -\partial_t(\phi-\xi)+H(x,D\phi)-H(x,D\xi)\right].
$$
We use again the fact that $ F^*_\alpha\geq 0$ and the convexity of $H$ to get
$$
0\geq  \iint_{{\mathcal O}_\ep} F^*_\alpha\left[ -\partial_t(\phi-\xi)+\lg H_p(x,D\xi), D(\phi-\xi)\rg\right].
$$
Since $\phi=\xi-\ep$ in $\partial {\mathcal O}_\ep$, we integrate by parts to obtain 
$$
0\geq  - \iint_{{\mathcal O}_\ep} (\phi-(\xi-\ep)) \left[- \frac{\partial}{\partial t} (F^*_\alpha)
+ \dive_x( F^*_\alpha H_p(x,D\xi))\right]
$$
Since $\phi-(\xi-\ep)> 0$ in ${\mathcal O}_\ep$, we must have 
$$
-\frac{\partial}{\partial t} (F^*_\alpha)
+ \dive_x( F^*_\alpha H_p(x,D\xi)) \geq 0 \; {\rm at}\; (t_0,x_0)\;, 
 $$
  which, after developing the terms, yield the desired inequality
$$
  {\mathcal G}\left(x_0, D_{t,x}\xi(t_0,x_0),D^2_{t,x}\xi(t_0,x_0)\right)\leq 0.
$$
The proof that $\phi$ is a supersolution can be treated along the same line (because we assume that  \eqref{CondSup} holds) and we omit it. \\

We already know that $\phi=\phi_T$ at $t=T$. It remains to check the boundary condition at $t=0$. We first prove that it holds in the viscosity generalized sense: for the subsolution part, let again $\xi$ be a smooth test function such that $\xi\geq \phi$ with an equality only at $(0,x_0)$. As before we can assume that inequality $-\partial_t\xi(t_0,x_0)+H(x,D\xi(t_0,x_0)))> 0$ holds. 
For any $\ep>0$, defining as above ${\mathcal O}_\ep=\{\xi-\ep<\phi\}$, 
$$
\phi_\ep= \phi\wedge (\xi-\ep), \; \alpha_\ep = (-\partial_t \xi+H(x,D\xi)){\rm 1}_{{\mathcal O}_\ep}+ \alpha {\rm 1}_{{\mathcal O}_\ep^c}\;, 
$$
we get:   
$$
\iint_{{\mathcal O}_\ep}F^*(x, -\partial_t \xi+H(x,D\xi)) -\int_{{\mathcal O}_\ep\cap \{t=0\}} (\xi-\ep)m_0
 \geq\iint_{{\mathcal O}_\ep}F^*(x,\alpha) -\int_{{\mathcal O}_\ep\cap \{t=0\}} \phi m_0
$$
We handle the terms $F^*$ as before to get 
$$
0\geq  \iint_{{\mathcal O}_\ep} F^*_\alpha\left[ -\partial_t(\phi-\xi)+\lg H_p(x,D\xi), D(\phi-\xi)\rg\right]
+ \int_{{\mathcal O}_\ep\cap \{t=0\}} (\xi-\ep-\phi)m_0
$$
We integrate by part and, since $\phi=\xi-\ep$  on $\partial {\mathcal O}_\ep\cap ((0,T)\times \T^d)$, we  obtain an extra boundary term
$$
0\geq  - \iint_{{\mathcal O}_\ep} (\phi-(\xi-\ep)) \left[- \frac{\partial}{\partial t} (F^*_\alpha)
+ \dive( F^*_\alpha H_p(x,D\xi))\right] + \int_{{\mathcal O}_\ep\cap \{t=0\}} (\xi-\ep-\phi)(m_0-F^*_\alpha)
$$
So either \eqref{kjhcvbss} holds, or we have $m_0-F^*_\alpha\geq 0$ at $(0,x_0)$, i.e., 
$-\partial_t\xi+H(x,D\xi)\leq f(x,m_0) $ at $(0,x_0)$ (because $f(x_0,m_0(x_0))\geq 0$). To complete the proof we just need to check that actually this last inequality holds. Assume on the contrary that there is $\kappa>0$ with  
\be \label{hypcontra}
-\partial_t\xi(0,x_0)+H(x_0,D\xi(0,x_0))> f(x_0, m_0(x_0))+\kappa.
\ee  We perturb 
$\xi$ into $\zeta(t,x)= \xi(t,x)+ \sigma(t)$ where $\sigma(0)=0$, $\sigma'(0)=\delta>0$ small and 
$\sigma''(0)= -R$ for $R$ large. Then, $\zeta\geq \xi \geq \phi$  for $t$ close to $0$. Since $-\partial_t\zeta+H(x_0,D\zeta)> f(x_0,m_0) $ at $(0,x_0)$
(from the choice of $\delta$ small), we have $m_0-F^*_\alpha(x_0,-\partial_t\zeta+H(x,D\zeta)) < 0$ and 
$\zeta$ must satisfy  \eqref{kjhcvbss}, so that
$$
\begin{array}{l}
\ds F^{*}_{\alpha,\alpha}\left[ -\partial_{tt}\xi-\sigma''(0)  +2\lg H_p, \partial_tD\xi\rg -\lg D^2\xi H_p, H_p\rg  \right] \\
\qquad \qquad  -F^{*}_{\alpha,\alpha}\lg H_p,H_x\rg - \lg F^*_{x,\alpha}, H_p\rg - F^*_\alpha\left[ Tr(H_{x,p})+Tr(H_{pp}D^2\xi)\right]\leq 0
\;\mbox{\rm at $(0,x_0)$,}
\end{array}
$$
 where,  to abbreviate the notation, we have set $F^{*}_{\alpha,\alpha}= F^{*}_{\alpha,\alpha}(x_0, -\partial_t \xi(0,x_0)-\sigma'(0)+H(x_0, D\xi(x_0)))$, etc... Since, by \eqref{hypcontra}, 
$-\partial_t \xi(0,x_0)-\sigma'(0)+H(x_0, D\xi(0,x_0))$ is larger than $\kappa/2$ for $\delta$ small (and independent of $R$), we have $F^{*}_{\alpha,\alpha}\geq \eta$ for some $\eta>0$ (independent of $R$) thanks to our assumption \eqref{HypSup2}. Now $-\sigma''(0) =R$ being arbitrarily large, we obtain a contradiction.

The proof that $\phi$ is a viscosity supersolution at $t=0$ can be handled in the same way, except for the boundary condition which requires additional explanation: let $\xi$ be a smooth test function such that $\xi\leq \phi$ with an equality only at $(0,x_0)$. Then one can show as before that either
$$
\min\left\{\  {\mathcal G}\left(x_0, D_{t,x}\xi(t_0,x_0),D^2_{t,x}\xi(t_0,x_0)\right)\ , \ -\partial_t\xi(t_0,x_0)+H(x_0,D\xi(t_0,x_0)) \ \right\}\geq 0,
$$
or $-\partial_t\xi(0,x_0)+H(x_0,D\xi(0,x_0))\geq f(x_0,m_0(x_0)) $.

We now argue by contradiction assuming that the second relation does not hold. 
Then, using the test function $\zeta(t,x)=\xi(t,x)-\sigma(t)$, where $\sigma$ is built as before, we must have 
\be\label{hjbqvlzn}
\min\left\{\  {\mathcal G}\left(x_0, D_{t,x}\zeta(t_0,x_0),D^2_{t,x}\zeta(t_0,x_0)\right)\ , \ -\partial_t\zeta(t_0,x_0)+H(x_0,D\zeta(t_0,x_0)) \ \right\}\geq 0.
\ee
Note that 
\be\label{hbqblf:f,b}
-\partial_t\zeta(t_0,x_0)+H(x_0,D\zeta(t_0,x_0)) = 
-\partial_t\phi(t_0,x_0)+H(x_0,D\phi(t_0,x_0)) +\delta \geq \delta>0 
\ee
because $\phi$ satisfies \eqref{CondSup}. By \eqref{hjbqvlzn} 
$$
\begin{array}{l}
\ds F^{*}_{\alpha,\alpha}\left[ -\partial_{tt}\xi+\sigma''(0)  +2\lg H_p, \partial_tD\xi\rg -\lg D^2\xi H_p, H_p\rg  \right] \\
\qquad \qquad  -F^{*}_{\alpha,\alpha}\lg H_p,H_x\rg - \lg F^*_{x,\alpha}, H_p\rg - F^*_\alpha\left[ Tr(H_{x,p})+Tr(H_{pp}D^2\xi)\right]\geq 0
\;\mbox{\rm at $(0,x_0)$,}
\end{array}
$$
 where $F^{*}_{\alpha,\alpha}= F^{*}_{\alpha,\alpha}(x_0, -\partial_t \xi(0,x_0)+\sigma'(0)+H(x_0, D\xi(x_0)))$, etc... 
Since \eqref{hbqblf:f,b} holds,  $F^{*}_{\alpha,\alpha}$ is positive thanks to assumption \eqref{HypSup2}. We can then let $R=-\sigma''(0)\to+\infty$ to get a contradiction. 

\end{proof}

\end{document}